\documentclass[12pt]{article}
\usepackage{latexsym}
\usepackage{amssymb}
\usepackage{amsmath}
\usepackage{mathrsfs}
\usepackage[pdftex]{graphicx}
\usepackage{rawfonts}
\input{prepictex}
\input{pictex}
\input{postpictex}
\usepackage[OT2,OT1]{fontenc}
\def\cyr{%
\renewcommand\rmdefault{wncyr}%
\renewcommand\sfdefault{wncyss}%
\renewcommand\encodingdefault{OT2}%
\normalfont
\selectfont}
\DeclareTextFontCommand{\textcyr}{\cyr}
\newcommand{\rad}{\varrho}

\def\bcp{\mathbb C \mathbb P}
\def\CC{\mathbb C}
\def\HH{\mathbb H}

\def\Vol{\mbox{Vol }}

\def\eea{\end{eqnarray*}}

\newtheorem{main}{Theorem}

\newtheorem{thm}{Theorem}[section]
\newtheorem{lem}[thm]{Lemma}
\newtheorem{prop}[thm]{Proposition}
\newtheorem{cor}[thm]{Corollary}
\newtheorem{defn}{Definition}

\newenvironment{proof}{\medskip \noindent
{\bf Proof.}}{\hfill \rule{.5em}{1em}
\\}

\def\ZZ{{\mathbb Z}}
\def\RR{{\mathbb R}}
\def\CP{{\mathbb C \mathbb P}}
\begin{document}

\title{On Optimal $4$-Dimensional Metrics}

\author{Claude LeBrun\thanks{Supported 
in part by  NSF grant DMS-0604735.}    ~and Bernard Maskit 
  }

\date{November 1, 2007}
\maketitle

\begin{abstract}
We completely determine, up to homeomorphism, which simply connected 
compact oriented $4$-manifolds  admit scalar-flat, anti-self-dual Riemannian metrics. 
The key
new ingredient is a proof that  the connected sum $\overline{\CP}_2\# \overline{\CP}_2\# \overline{\CP}_2\# \overline{\CP}_2\# \overline{\CP}_2$
 of five reverse-oriented complex projective planes admits such metrics. 
\end{abstract}

\section{Introduction}

Marcel Berger \cite{bercent} credits the late  
 Ren\'e Thom  with  the following vague, but  fundamental,  question:
 \begin{quote}
 {\sl Does every smooth compact manifold admit a best metric?}
 \end{quote}
Berger eventually proposed a more precise version of  the problem by  asking 
which smooth 
compact $n$-manifoldsx $M$, $n\geq 3$,  admit  Riemannian  metrics $g$ 
which  are {\em as flat as possible}, in the sense that they minimize 
the scale-invariant functional 
$$
g~\longmapsto ~ 
{\mathcal K}(g) =  \int_{M}|{\mathcal R}_{g}|_{g}^{n/2}d\mu_{g}, $$
where ${\mathcal R}$ denotes the Riemann curvature tensor, 
$|{\mathcal R} |$ is its point-wise norm with respect to the metric, and
 $d\mu$ is the  $n$-dimensional  volume measure determined by the metric.
 The following terminology is then used to describe 
metrics which are ``best'' in this precise sense: 
  \begin{defn}\label{defopt} 
 Let $M$ be a smooth compact $n$-dimensional manifold, $n\geq 3$.
 A smooth Riemannian metric $g$ on $M$ is said to be an 
 {\em optimal metric} if it is an absolute minimizer of the above-defined 
 functional $\mathcal K$,
in the sense   that 
 $${\mathcal K}(\tilde{g}) \geq {\mathcal K}(g)$$
 for every smooth Riemannian metric $\tilde{g}$ on $M$. 
 \end{defn}

   Berger's  chief motivation  for  Definition \ref{defopt} seems to have been  
   that,  as explained  in \S  \ref{rud} below,  
   Einstein metrics on compact $4$-manifolds are
 optimal in this sense. But 
 this fact also    shows that  dimension $4$ enjoys 
 a  peculiar, privileged  status for the problem. Indeed, notice that, for any integer $m\geq 2$, 
 there is a {\em non-optimal}  Einstein metric $g$   on  $S^3 \times S^{m}$,
 given by the product of the standard `round' metrics of 
    radii $\sqrt{2}$ and $\sqrt{m-1}$, respectively. Since $g$ is not flat, we obviously havee 
    ${\mathcal K} (g) > 0$. However, there is a sequence  of 
homogeneous  metrics  on  $S^3 \times S^{m}$ with ${\mathcal K}\searrow  0$, 
as may be   constructed   by shrinking $g$
  along the fibers of the Hopf fibration  $S^3 \times S^{m}\to \CP_1 \times S^{m}$.
  This shows that there are non-optimal Einstein metrics in any dimension $\geq 5$. Since
  essentially the same argument also shows that the constant-curvature metric  on
  $S^3$ isn't optimal either, the very special status of dimension four is now manifest. 
  
  However, while every $4$-dimensional Einstein metrics is optimal, 
  not every  $4$-dimensional  optimal metric is Einstein. This
  fact, which  is of fundamental importance for our purposes here, 
   now merits a careful explanation. 
  
 \subsection{Optimal Metrics in Dimension Four}
 \label{rud} 
 
Why, then, is dimension four so special? To a large extent, this is because 
the bundle of $2$-forms on an oriented Riemannian $4$-manifold
$(M,g)$ can be invariantly decomposed as a direct sum
\begin{equation} 
\Lambda^2 = \Lambda^+ \oplus \Lambda^- , 
\label{deco} 
\end{equation}
of   the $(\pm 1)$-eigenspaces $\Lambda^\pm$ of the Hodge
$\star$ operator. 
Since the 
Riemann curvature tensor ${\mathcal R}$ may be thought of   as a linear map 
$\Lambda^2 \to \Lambda^2$,  the decomposition   (\ref{deco})
therefore  allows us to think of   ${\mathcal R}$ as consisting of four blocks: 
\begin{equation}
\label{curv}
{\mathcal R}=
\left(
\mbox{
\begin{tabular}{c|c}
&\\
$W_++\frac{s}{12}$&$\mathring{r}$\\ &\\
\cline{1-2}&\\
$\mathring{r}$ & $W_-+\frac{s}{12}$\\&\\
\end{tabular}
} \right) . 
\end{equation}
Here $W_\pm$ are the trace-free pieces of the appropriate blocks,
and  are  called the
self-dual and anti-self-dual Weyl curvatures, respectively; these pieces  
of the curvature tensor are {\em conformally invariant}, in the sense that 
they are unchanged if $g$ replaced by  $u^2g$, where $u$ is an arbitrary smooth positive
function. 
The scalar curvature  $s$ is understood to act by scalar multiplication,
whereas the   
     trace-free Ricci curvature
$\mathring{r}=r-\frac{s}{4}g$ 
acts on 2-forms by
$$\varphi_{ab} \longmapsto ~ {\textstyle\frac{1}{2}}\Big[
\mathring{r}_{a}^c{\varphi}_{cb} -
\mathring{r}_{b}^c{\varphi}_{ca}\Big].$$
In terms of this decomposition of the curvature tensor of an arbitrary metric
$g$ on a compact oriented $4$-manifold $M$, 
the  generalized Gauss-Bonnet theorem expresses the 
Euler characteristic of $M$ as 
\begin{equation}
\label{gb}
\chi (M) = \frac{1}{8\pi^2} \int_M \left(\frac{s^2}{24}+
 |W_+|^2+|W_-|^2-\frac{|\mathring{r}|^2}{2}  \right)d\mu_g , 
\end{equation}
while  the 
Hirzebruch signature theorem  allows us to express the 
signature of $M$ as 
\begin{equation}
\label{sig}
\tau (M) = \frac{1}{12\pi^2}\int_M  \left(
 |W_+|^2-|W_-|^2  \right)d\mu_g ~ .
\end{equation}
Since our curvature  functional  becomes  
$$
 {\mathcal K}(g) = \int_M |{\mathcal R}|^2d\mu_g = \int_M \left(\frac{s^2}{24}+
 |W_+|^2+|W_-|^2+\frac{|\mathring{r}|^2}{2}  \right)d\mu_g ~,
 $$
 in dimension $4$, 
we therefore have
$$
{\mathcal K}(g)  = 8\pi^2 \chi (M) + \int_M |\mathring{r}_g|^2d\mu_g~, 
$$
implying   Berger's observation that  any Einstein metric  $g$  
minimizes $\mathcal K$.
However, {\em not every optimal metric is Einstein}, even in dimension 
four. For example,  the Gauss-Bonnet and signature formul{\ae}
imply  that 
$$
{\mathcal K}(g)  =
 -8\pi^2 (\chi + 3\tau )(M) + 2\int_M \left(\frac{s^2}{24}+2|W_+|^2\right) d\mu_g 
 $$
so,  as was perhaps first  observed by Lafontaine \cite{laf}, 
 another class of 
minimizers is given by the 
metrics for which both $W_+$ and $s$ are identically
zero. We will indicate  the latter class of metrics
by means of the  following  terminology: 
 \begin{defn}
If $M$ is a smooth oriented $4$-manifold, a Riemannian metric $g$
on $M$ is  said to be {\em anti-self-dual} (or, for brevity,  {\em ASD}) 
if its self-dual Weyl curvature is identically zero: 
$$W_+\equiv 0.$$
A metric $g$ is  called 
{\em scalar-flat}  (or, more briefly, {\em SF}) 
if it satisfies
$$s\equiv 0.$$
Finally, we  say that $g$ is {\em scalar-flat anti-self-dual} (or {\em SFASD}) if it satisfies both of 
these conditions. 
\end{defn}

Notice that the ASD condition is conformally invariant; that is, if $g$ is ASD, so is 
$u^2 g$, for any positive function $u$. By contrast, if $g$ has scalar curvature
$s$, then $\hat{g}= u^2 g$ has scalar curvature $\hat{s}$ determined by 
$$\hat{s}u^3 = ( 6\Delta + s) u,$$
where $\Delta= - \nabla\cdot \nabla$  is the Laplace-Beltrami operator of $g$.
In particular, a conformal class on a compact manifold can contain at
most one scalar-flat metric, up to overall constant rescaling. 

An immediate corollary of the above ideas is that any 
 locally conformally flat, scalar-flat metric on a compact $4$-manifold 
 is optimal. This fact alone provides  many examples of 
non-Einstein optimal metrics on compact $4$-manifolds with infinite
fundamental group; for example, 
   the   product metric on $S^2 \times \Sigma$,
 where $\Sigma$ is a compact surface 
of genus $\geq 2$
 equipped with a choice of hyperbolic
(Gauss curvature $-1$)  metric, and 
where the $2$-sphere $S^2$ is taken to be equipped with its standard 
(Gauss curvature $+1$) metric.  However, examples produced by this trick
can never be  simply connected \cite{kuiper}.  Nonetheless, simply connected SFASD 
manifolds do exist in considerable profusion, and, as we shall  explain in the next
subsection,  the entire purpose of this paper is to provide a complete topological 
classification of those simply connected $4$-manifolds which 
admit optimal metrics of this special type.

\subsection{The Main Result}

The existence of an SFASD metric places tight constraints on 
the topology of a $4$-manifold. Indeed, the following result was pointed out
in \cite{loptimal}: 
 
\begin{prop}\label{baker} 
A   smooth compact  simply connected 
$4$-manifold $M$  admits scalar-flat anti-self-dual 
 metrics only if 
\begin{itemize}
\item $M$ is homeomorphic to $k \overline{\bcp}_2$ for some $k \geq 5$; or 
\item $M$ is diffeomorphic to $\bcp_2\#k \overline{\bcp}_2$ for some  $k \geq 10$; or else 
\item $M$ is diffeomorphic to $K3$.
\end{itemize}
\end{prop}

The key point   \cite{lsd} is that a self-dual harmonic $2$-form on a compact   SFASD 
$4$-manifold must be parallel; thus, either the intersection form on $H^2$ must be negative-definite, 
or else   the metric must be  K\"ahler. Moreover, as observed by Lafontaine \cite{laf}, 
our Gauss-Bonnet-type formul{\ae} imply that any SFASD manifold  must
 satisfy  $2\chi + 3\tau  \leq 0$, with equality only if
 the metric is locally hyper-K\"ahler. The Proposition therefore
 follows from the 
  celebrated results of Donaldson 
\cite{donaldson}
and Freedman \cite{freedman} on the topology of smooth $4$-manifolds,
together with an elegant plurigenus vanishing argument due to Yau \cite{yauruled}. 

\bigskip 
 
 \noindent 
 The purpose  of the present paper is to prove a
 near-converse of the above Proposition: 
 
 \bigskip 
 
 \begin{main}\label{able}
 A   smooth compact 
 simply connected  $4$-manifold $M$  admits scalar-flat anti-self-dual 
 metrics if 
 \begin{itemize}
 \item  $M$ is diffeomorphic to  $k \overline{\bcp}_2$ for  some $k  \geq 5$; or 
\item $M$ is diffeomorphic to $\bcp_2\#k \overline{\bcp}_2$ for some  $k \geq 10$; or  
\item $M$ is diffeomorphic to $K3$.
\end{itemize}
In particular, each of these manifolds carries an optimal metric. 
 \end{main}

Combining  Proposition \ref{baker} and Theorem \ref{able}, we thus have:

\begin{cor} A  compact  simply connected topological $4$-manifold $M$ carries  a
smooth structure for which there is a compatible SFASD  metric
$g$ iff 	$M$ is homeomorphic to $K3$, $\bcp_2\#k \overline{\bcp}_2$,   $k \geq 10$, or 
$k \overline{\bcp}_2$,  $k  \geq 5$. 
\end{cor}

 It is still  unknown whether any $k \overline{\bcp}_2$ 
 admits non-standard smooth structures, so the degree to which 
 Theorem \ref{able} represents a true  converse of Proposition \ref{baker}
remains poorly  understood. 
Note, however, that 
each of the topological manifolds $K3$ and  $\bcp_2\#k \overline{\bcp}_2$,   $k \geq 10$,
admits infinitely many exotic smooth structures, and  Proposition 
\ref{baker} asserts that there can never be an SFASD metric compatible with any of these. 
Moreover, for large classes  of 
such   exotic smooth structures, related arguments even    show \cite[Theorem B]{loptimal}
that these manifolds do not admit 
compatible {optimal} metrics of {\em any} kind.

Now, as indicated above,  the proof of  Proposition \ref{baker} hinges on the fact that 
any self-dual harmonic $2$-form 
on an SFASD $4$-manifold must be parallel. For this reason, 
the only such metrics on $K3$ are hyper-K\"ahler, and the corresponding
existence assertion in Theorem \ref{able} therefore follows
from Yau's solution \cite{yauma} of the Calabi conjecture. Similarly, 
the   SFASD metrics on $\bcp_2\#k \overline{\bcp}_2$ are precisely the scalar-flat
K\"ahler metrics on these spaces; for $k\geq 14$, the existence of 
such metrics was first shown by Kim, Pontecorvo, and the first author \cite{klp},
a result which was later improved to $k\geq 10$ by Rollin and Singer \cite{rollsing}. 
On the other hand, the existence of SFASD metrics on $k\bcp_2$, $k\geq 6$,
was first shown in \cite{loptimal}. The main goal of the present paper is therefore
to improve the last assertion in order to include the case of $k=5$;  however, 
we will in fact obtain a  simple, 
unified construction of such metrics for all $k\geq 5$ at no added cost.  We now begin by providing a brief overview 
of the entire construction.

\subsection{Strategy of the Proof}
\label{stratego}

In order to prove Theorem \ref{able}, we begin by choosing some integer $\ell \geq 3$, 
and then consider the two oriented conformally-flat orbifolds 
given by 
$S^4/\ZZ_2$ and $S^4/\ZZ_\ell$,
 where the relevant actions of 
$\ZZ_2$  and $\ZZ_\ell$ on the quaternionic projective line  ${\mathbb H \mathbb P}_1=S^4$
are respectively generated by 
$$
\left(\begin{array}{cc}1 & 0  \\0 & -1 \end{array}\right)
~~~~~~\mbox{and}~~~~~~
\left(\begin{array}{cc}1 & 0  \\0 & e^{2\pi i/\ell}  \end{array}\right).
$$
Both of these orbifolds have two singular points, corresponding to 
$[1:0]$ and $[0:1] \in {\mathbb H \mathbb P}_1$. 
However, while the two singularities of $S^4/\ZZ_2$ 
are on an equal footing, the two singularities of $S^4/\ZZ_\ell$  are 
instead 
mirror reflections of each other, corresponding to local actions on 
the quaternions 
$\HH= \RR^4$   respectively generated  by
$$
q\longmapsto e^{2\pi i/\ell} q
~~~~~~\mbox{and}~~~~~~
q\longmapsto q e^{-2\pi i/\ell} .    
$$

We now form the connected sum of these two orbifolds by 
removing a small  round ball from  the non-singular region of each, and then 
identifying the resulting $S^3$-boundaries via a reflection:  
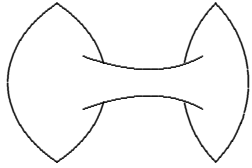
\begin{center}
\mbox{
\beginpicture
\setplotarea x from 0 to 290, y from 0 to 60
\ellipticalarc  axes ratio 3:2 90 degrees from 110 60 
center at 155 30
\ellipticalarc  axes ratio 3:2  -35 degrees from 110 60 
center at 65 30
\ellipticalarc  axes ratio 3:2  35 degrees from 110 0 
center at 65 30
\circulararc  -90 degrees from 170 60 
center at 140 30
\circulararc  35 degrees from 170 60 
center at 200 30
\circulararc  -35 degrees from 170 0 
center at 200 30
{\setquadratic 
\plot 120 40    145 35    165 40   /
\plot 120 20    145 25    165 20  /
}
\endpicture
}
\end{center}
The resulting  orbifold $V=(S^4/\ZZ_2)\# (S^4/\ZZ_\ell)$ is thereby endowed with 
an anti-self-dual  orbifold conformal structure --- indeed, with a locally conformally flat one. 

Now each  singularity of  the orbifold $V$
 looks like one familiar from the 
theory of complex algebraic surfaces \cite{bpv}, and  each therefore has a 
 so-called  minimal resolution. We may thus  desingularize $V$
to  obtain a smooth $4$-manifold. 
Strictly on the  level of smooth topology, 
 this process amounts to 
 replacing a neighborhood $B^4/\Gamma$ of  each singular point
 with a standard $4$-manifold plug bounded by  the same Lens space $S^3/\Gamma$. 
For the two 
$\Gamma = \ZZ_2$ singularities, the interior of our plug is just the 
degree $-2$ complex line bundle over $S^2 = \bcp_1$. Similarly, 
for the  singularity arising from the action of $\Gamma= \ZZ_\ell$ generated by 
$$q \mapsto e^{2\pi i/\ell} q, $$
the interior of our plug is just the degree $-\ell$ complex line bundle over $S^2$. 
  Finally, for  the 
  singularity arising from the action of $\Gamma= \ZZ_\ell$ generated by 
$$q \mapsto q e^{-2\pi i/\ell} , $$
the interior of  
our plug is a union of $\ell-1$ copies of the 
degree $-2$ complex line bundle over $S^2$,
plumbed together  
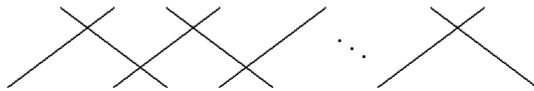
\begin{center}
\mbox{
\beginpicture
\setplotarea x from 0 to 200, y from -10 to 40
\put {$\ddots$} [B1] at 130 10
{\setlinear
\plot   0 0  40 30  /
\plot   20 30  60 0   /
\plot   40 0  80 30  /
\plot   60 30  100 0   /
\plot   80 0  120 30   /
\plot   140 0  180 30   /
\plot   160 30  200  0   /
}
\endpicture
}
\end{center}
in the manner  dually indicated  by 
 the 
Dynkin diagram $A_{\ell -1}$: 
 \begin{center}
\mbox{
\beginpicture
\setplotarea x from 0 to 100, y from -10 to 10 
\put {$\bullet$} [B1] at 0 0 
\put {$\bullet$} [B1] at 15 0 
\put {$\bullet$} [B1] at 30 0 
\put {$\bullet$} [B1] at 45 0 
\put {$\bullet$} [B1] at 60 0 
\put{$\cdots$}[B1] at 73 0
\put {$\bullet$} [B1] at 85 0 
\put {$\bullet$} [B1] at 100 0 
{\setlinear  
\plot 0  3 60 3 /
\plot 85  3 100 3 /
}
\endpicture
}
\end{center}
Now it is  known \cite{mcp2} that if this desingularization process is separately 
applied to $S^4/\ZZ_2$ and $S^4/\ZZ_\ell$, the 
resulting manifolds are respectively  $2\overline{\bcp}_2$ and  $\ell\overline{\bcp}_2$.
Thus, the $4$-manifold we obtain from $V= (S^4/\ZZ_2)\# (S^4/\ZZ_\ell)$
by this process will exactly be  $M= (\ell +2) \overline{\bcp}_2$. In particular, 
the $\ell = 3$ case will gives us the manifold $M= 5 \overline{\bcp}_2$,    where 
the existence of  optimal metrics is the key  issue at stake. 

Now each  plug we have used to replace a singularity 
admits 
an asymptotically locally Euclidean (ALE) scalar-flat anti-self-dual 
(SFASD) metric \cite{EH,gibhawk,hitpoly,lpa}. 
The ALE property means that the complement of a compact
set is diffeomorphic to 
 $(\RR^4-B^4)/\Gamma$ in such a manner that the metric 
takes the form 
$$g_{jk} =  \delta_{jk} + O (\varrho^{-2}),$$
where $\varrho$ is the Euclidean radius.
Since multiplying such a metric   by a tiny positive number therefore yields 
 a space whose geometry is macroscopically indistinguishable from 
the flat orbifold  $\RR^4/\Gamma$, one might thus hope to 
find  anti-self-dual metrics on  $M= (\ell +2) \overline{\bcp}_2$ by 
grafting these ALE metrics onto the conformally flat conformal structure
of the non-singular region of $V$, and then perturbing the grafted metric
so as to once again solve the anti-self-duality equation. As we shall see, 
this does indeed work. Our approach to this aspect
follows  the direct analytic approach first proposed  by 
Floer \cite{floer,kovsing,rollsing,tasd}. We remark  in passing that 
the alternative  of  a twistor approach \cite{AHS,pnlg} to this problem
 might seem particularly  tempting, but that the requisite 
  stack-theoretic generalization of the 
 Donaldson-Friedmann
construction  \cite{DF} has yet to be put on a solid footing;  however, see 
\cite{klp,lebsing2,zhouthesis} for   evidence  that such an approach 
should indeed  be tractable.  Note that this gluing procedure depends
on a  general vanishing result (Theorem \ref{sesame})
which had previously been asserted in  \cite[Theorem 8.4]{kovsing};
here, our contribution is  to point out and repair a    gap in the earlier proof.

Now, we have already noted that $V$ carries a conformally flat structure. However, 
this conformally flat structure is far from unique. To the contrary, 
 as we will see in \S \ref{moduli} below, the theory of 
 Kleinian groups \cite{maskit} provides us with a connected family of conformally flat structures
 for which the corresponding limit sets have  Hausdorf dimension 
 varying from nearly zero to nearly two. By a result of Schoen-Yau \cite{schyaudim}, 
 this implies that some of these conformal structures are represented by 
 metrics of positive scalar curvature, while others are represented by   metrics of 
negative scalar curvature. The punch line of our story 
is that the above gluing construction can be carried out uniformly
in these additional parameters, and,  setting 
 $k=\ell+2$, a gluing argument 
  suggested by the work of  Dominic Joyce \cite{joyscal}   allows us  to show 
the existence of $1$-parameter  family of ASD metrics on $M=k \overline{\bcp}_2$ 
for which the scalar curvature changes sign in an analogous  manner:

\begin{main}\label{charlie}
For any integer $k\geq 5$, the smooth  oriented $4$-manifold 
$$k\overline{\bcp}_2=\underbrace{\overline{\bcp}_2\# \cdots \# \overline{\bcp}_2}_k$$ 
admits
a smooth $1$-parameter family of anti-self-dual  conformal metrics $[g_t]$, 
$t\in [-1,1]$,
such that $[g_1]$ contains a metric of positive scalar curvature, while
$[g_{-1}]$ contains a metric of negative scalar curvature.
\end{main}

Now, because the lowest eigenvalue of the Yamabe Laplacian depends continuously 
on $t$, it is then easy to show that there must be  some $t\in [-1,1]$
such that $[g_t]$ contains a metric with scalar curvature $s\equiv 0$. 
In particular, this then shows that,  for any $k\geq 5$,  the 
$4$-manifold  $M=k\overline{\bcp}_2$ carries
a scalar-flat anti-self-dual metric.   In light of our previous discussion, 
Theorem \ref{able}  is therefore an immediate 
consequence.

\section{Uniformizing the Orbifold $V$} \label{uniform} 

The orbifold $V$ of \S \ref{stratego} is actually expressible as a global quotient 
$X/D_{\ell}$, where the locally conformally flat $4$-manifold 
$$X= \underbrace{(S^3\times S^1) \# \cdots \# (S^3\times S^1)}_{\ell -1}$$
may be  constructed by connecting two $4$-spheres with $\ell$ tubes
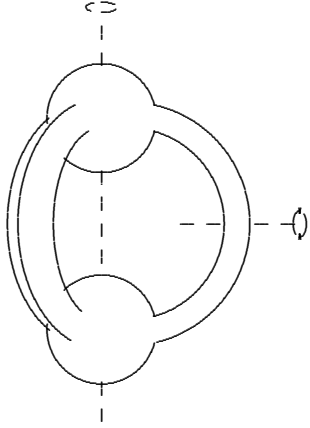
\begin{center}
\mbox{
\beginpicture
\setplotarea x from 0 to 300, y from -20 to 140
\circulararc  172 degrees from 170 105 
center at 150 100
\circulararc  -120 degrees from 170 95 
center at 150 100
\circulararc  -172 degrees from 170 15 
center at 150 20
\circulararc  120 degrees from 170 25 
center at 150 20
\circulararc  -154 degrees from 170 105 
center at 160 60
\circulararc  -149 degrees from 170 95 
center at 160 60
\ellipticalarc  axes ratio 2:3  122  degrees from 130 100
center at 145 60
\ellipticalarc  axes ratio 2:3  140 degrees from 140 105
center at 150 60
\ellipticalarc  axes ratio 1:2  140 degrees from 145 95
center at 150 60
\ellipticalarc axes ratio 1:2 140 degrees from 225 65
center at 225 60
\ellipticalarc axes ratio 1:2 140 degrees from 225 55
center at 225 60
\arrow <2pt> [.1,.3] from 225 55 to 225 56
\arrow <2pt> [.1,.3] from 225 65 to 225 64
\ellipticalarc axes ratio 3:1 70 degrees from  150 140
center at 150 142 
\ellipticalarc axes ratio 3:1 60 degrees from  155 143
center at 150 142 
\ellipticalarc axes ratio 3:1 110 degrees from  147 143.7
center at 150 142 
{\setdashes
\plot 150 120 150 140 /
\plot 150 80 150 40 /
\plot 150 0 150 -20 /
\plot 180 60 200 60 /
\plot 223 60 205 60 /
}
\endpicture
}
\end{center}
and where the action of the dihedral group 
$D_\ell$  on $X$ is then generated by a cyclic permutation of the 
handles and by an interchange of  the two $4$-spheres. This can be seen  most directly 
by first recognizing $V= (S^4/\ZZ_2) \# (S^4/\ZZ_\ell)$ as a $\ZZ_\ell$-quotient 
of $S^4 \# \ell   (S^4/\ZZ_2)$, and then observing that this latter space can
in turn be thought of  as $X/\ZZ_2$. 

\subsection{The Kleinian Group Picture} \label{naive} 

Let us now explicitly realize $X$ and $V$ as quotients of 
regions of the $4$-sphere by groups of M\"obius transformations. 
Clearly,  a fundamental region in the universal cover of $X$ can 
be taken to be  $S^4= \HH\cup \{ \infty\}$   minus the $\ell$  open balls
of radius $\varepsilon$ centered at the $\ell^{\rm th}$ roots of unity 
in $\CC\subset \HH$, 
together with the inversion of this complement into the $\varepsilon$-ball centered at $1$:
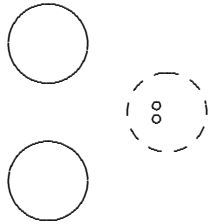
\begin{center}
\mbox{
\beginpicture
\setplotarea x from 0 to 300, y from 0 to 120
\circulararc  360 degrees from 176 64
center at 176 62.5
\circulararc  360 degrees from 175.9 56
center at 176 57.5
\circulararc  360 degrees from 135 101
center at 135 86
\circulararc  360 degrees from 135 49
center at 135 34
\setdashes 
\circulararc  360 degrees from 180 75 
center at 180 60
\endpicture
}
\end{center}
The universal cover $\tilde{X}$ of $X$ is thus realized as $S^4$ minus a 
 Cantor set  $\Lambda\subset \HH$, where $\pi_1(X) = \ZZ \ast \cdots \ast \ZZ$ 
acts  via a certain representation $\pi_1(X) \to PL (2, \HH )$.  Similarly, we have 
$V= \tilde{X}/(\ZZ_2\ast \ZZ_\ell)$, where $\ZZ_\ell$
is generated by the rotation $$\beta(q)= e^{2\pi i/\ell}q,$$ and where
$\ZZ_2$ is generated by the inversion $$\alpha (q) = 1+ \varepsilon^2 (q-1)^{-1}.$$
Note that $\pi_1 (X)$ is the subgroup of $\ZZ_2\ast \ZZ_\ell$ generated by
$(\alpha \beta)^2$ and its conjugates, and so is exactly the kernel of 
the obvious surjective homomorphism $\ZZ_2\ast \ZZ_\ell\to D_\ell$. 
Since $\alpha$ and $\beta$ actually actually lie in $PSL (2, \CC) < PL (2, \HH )$, 
the relevant representations of $\ZZ_2\ast \ZZ_\ell$ and 
$\pi_1 (X) \triangleleft  \ZZ_2\ast \ZZ_\ell$ are actually
$PSL (2, \CC)$-valued, and the limit set $\Lambda$ is thus a subset 
of the plane $\CC  \subset \HH$. 

\subsection{The Limit Set $\Lambda$}

The limit  set $\Lambda$ comes, by construction, with a hierarchy 
of coverings by disks in $\CC$.  At the crudest level, it is contained in a union of
$\ell$ disks of radius $\varepsilon$, where we are allowed to choose
any positive $\varepsilon < \sin (\pi/\ell)$. 
We now express this choice as 
$$\varepsilon = \frac{\sin (\pi/\ell)}{C+1}$$ 
for some  real number $C> 0$. Each disk in the successive layers of the heirarchy 
contains $\ell -1$ times as many disks as in the previous layers, and 
each disk in any  layer  has radius smaller than that of 
some disk in the  preceding layer by a factor of better than $C^{-2}$. In particular, 
$\Lambda$ is contained in a union of $\ell (\ell -1)^N$ disks of radius less
than  $C^{-2N}$, and so has $d$-dimensional Hausdorff measure 
$< \ell [(\ell -1)/C^{2d}]^N$.   Taking $N\to \infty$, we thus see  that 
 the Hausdorff measure of $\Lambda$ is  zero
in any dimension $> [\log (\ell -1)]/2\log C$. For fixed $\ell$, it follows that 
that the Hausdorff dimension $\dim \Lambda$ can be taken arbitrarily close
to $0$ by taking $\varepsilon$ to be sufficiently small.

This now implies that some of the constructed
conformal structures on $V$ are represented by conformally flat
orbifold metrics of positive scalar curvature. 
Indeed, let us recall a 
 remarkable result of Schoen and Yau \cite{schyaudim}, together with a slight refinement due to 
   Nayatani \cite{shinpat}: 

\begin{lem}[Schoen-Yau, Nayatani] 
\label{schauen}
Let (M,[g]) be a compact, locally conformally flat 
$n$-manifold, $n\geq 3$,  which can be uniformized as $$M=\Omega/G,$$ where 
$G\subset SO^\uparrow (n+1, 1)$ is a Kleinian group
and where $\Omega\subset S^n$ is the region of discontinuity of $G$. Let 
$g\in [ g]$ be a  metric on $M$ in the fixed conformal class for which 
the scalar curvature $s$ does not change sign. Assume that  the limit set $\Lambda$ of $G$
is infinite, and let $\dim (\Lambda ) > 0$ denotes its 
Hausdorff dimension.
Then
\begin{eqnarray*}
s > 0 & \Longleftrightarrow & \dim (\Lambda ) < \frac{n}{2}-1 \\
s = 0 & \Longleftrightarrow & \dim (\Lambda ) = \frac{n}{2}-1 \\
s < 0 & \Longleftrightarrow & \dim (\Lambda ) > \frac{n}{2}-1  . 
\end{eqnarray*} 
\end{lem}

Now any conformal class on any compact manifold contains
metrics for which the scalar curvature does not change sign;
moreover, as observed by Trudinger \cite{trud}, this sign coincides, for any background metric 
$g$,  with 
that of the lowest eigenvalue $\lambda$ of the Yamabe Laplacian,
given in dimension $4$ by 
$$\Delta_g + \frac{s_g}{6}~;$$
indeed, if $u$ is an eigenfunction with eigenvalue $\lambda$, then 
$u$ is nowhere zero by the minimum principle, and 
$\hat{g}=u^2g$ then has scalar curvature
$$\hat{s}= u^{-3}6(\Delta_g + \frac{s_g}{6})u= 6\lambda u^{-2}$$
which everywhere has the same sign as $\lambda$. Let us apply
this to $$X= \underbrace{(S^3\times S^1) \# \cdots \# (S^3\times S^1)}_{\ell -1}~.$$
By averaging, we may begin by 
choosing our background metric $g$ on $X$ to be invariant 
under the action of the finite group $D_\ell$. Because the minimum principle
implies that the lowest eigenvalue $\lambda$ of the Yamabe Laplacian has
multiplicity one, the corresponding lowest eigenfunction $u$ is unique
up to an overall multiplicative constant, and so must be invariant under 
the finite group of isometries $D_\ell$. Thus $\hat{g}=u^{2}g$ may either
be viewed as a 
$D_\ell$-invariant conformally flat metric on $X$,  or else as a
conformally flat orbifold metric on $V=X/D_\ell$. 
By Lemma \ref{schauen},  the sign of the scalar curvature of $\hat{g}$ will therefore 
be positive iff
the Hausdorff dimension of  the limit set $\Lambda$ of our 
$\ZZ_2 \ast \ZZ_\ell$ action is less than $1$. Since we have just seen that this can be 
accomplished by taking the parameter $\varepsilon$ to be sufficiently small, it
follows that some of our conformal classes are indeed represented by 
conformally flat orbifold metrics on $V$ which have positive scalar curvature,
as would also be predicted by more elementary considerations.

Our real goal is to show that the above  metrics on $V$ 
can be  
continuously deformed through locally conformally flat metrics in order to yield  metrics
of negative scalar curvature. We will do this by 
deforming the Kleinian groups $\ZZ_2 \ast \ZZ_\ell \hookrightarrow PSL (2, \CC )$
considered above 
into   ones whose limit sets have  Hausdorff dimension greater than $1$. This problem  is intimately tied to the
theory of non-classical Schottky groups, and its solution will require the entire
 next section of the paper.

\section{Deformations with large limit sets} \label{moduli}
The main point of this section is to show that there are quasi-conformal deformations of the
Kleinian group $G$ of \S \ref{naive}
whose limit sets have Hausdorff dimension arbitrarily close to 2. The proof, which uses ideas going back to Bers \cite{Bers:BdrysI} and Maskit \cite{Maskit:BdrysII}, can be easily generalized to a larger class of Kleinian groups. Rather than work out a detailed formulation of this more general theorem, we content ourselves with assuming that $G $ is a combination theorem free product of an
elliptic cyclic group of order
$m$, and an elliptic cyclic group of order $\ell  \ge m$, where $\ell  \ge 3$. That is, $G \cong
\ZZ_m \ast \ZZ_\ell$ is the free product of an element $\alpha$ of order $m$, and an element $\beta$ of order $\ell$. For our application, we of course only need the case of $m=2$, and it
is the case of  $\ell=3$ which is of primary interest here.

\subsection{Holomorphic coordinates}\label{coordinate} 
We note that, since $\ell   \ge 3$, the fixed points of $\beta$ can be distinguished. That is, there is no M\"obius transformation conjugating $\beta$ into itself and interchanging its fixed points. We label one of these fixed points as positive. (For example, if we conjugate $\beta$ so that its positive fixed point is at 0, while its negative fixed point is at $\infty$, then this conjugate has the form $z\mapsto e^{2\pi i/\ell }z$.)

We normalize $G $ so that $\alpha$ has fixed points at 0 and $\infty$, and so that the positive fixed point of $\beta$ is at 1; if the order $m$ of $\alpha$ is greater than 1, then we put the positive fixed point of $\alpha$ at 0. Let $z_0$ be the negative fixed point of $\beta$. For every $z\ne 1$, we define the homomorphism $\phi_z:G \to PSL(2,{\mathbb C})$ given by $\phi_z(\alpha) = \alpha$ and $\phi_z(\beta)$ is the Mobius transformation of order $\ell $ with positive fixed point at 1 and negative
fixed point at $z$.

We observe that if $z = 0$, or $z = \infty$, then $\phi_z(G)$ contains parabolic elements and has a fixed point $z$. We also note that if $z$ is real and negative, then the axes of $\phi_z(\alpha)$ and $\phi_z(\beta)$ intersect at a point in the upper half space; this point is necessarily a fixed point of $\phi_z(G)$. It was observed by Chuckrow \cite{Chuckrow:Schottky} that no point on the algebraic convergence boundary of the deformation spaace of $G$ has a fixed point in ${\mathbb H}^3$ or on the
sphere at infinity.

We denote the complex plane, punctured at 1, and with the negative real axis deleted, by $\CC^\bullet$. The homomorphism $\phi_z$ defines an embedding of $\CC^\bullet$ into the space of homomorphisms of $G $ into $PSL(2,{\mathbb C})$.

 Let ${\mathcal T} = {\mathcal T}(G )$ denote the (quasi-conformal) deformation space of $G $; where the quasi-conformal mappings are as usual normalized so as to fix $(0,1,\infty)$. With this normalization, every quasi-conformal mapping representing an element of ${\mathcal T}$ conjugates $\alpha$ into itself and preserves the positive fixed point of $\beta$. Since no element of the closure of ${\mathcal T}$ can conjugate $G$ into a group with a fixed point, we can regard the closure of ${\mathcal T}$
as a subset of $\CC^\bullet$.

It was shown by Ahlfors and Bers \cite{A-B:Beltrami} that the image of a point
under a family of
quasi-conformal mappings is a holomorphic function of parameters for the family; it follows that if we regard ${\mathcal T}$ as being endowed with the complex structure defined by the projection from its universal covering space, the Teichm\"uller space of Riemann surfaces of genus 0 with four punctures, then this embedding of ${\mathcal T}$ into $\CC^\bullet$ is holomorphic.

The parameter $z$ appears to depend on our normalization. However, if we write the cross-ratio of four points as
\begin{equation}\label{Eqn:zfromt}
(a,b;c,d) = \frac{(a-c)(b-d)}{(a-d)(b-c)},
\end{equation}
then  $z = (z,1;0,\infty)$.

\subsection{Geometric coordinates}

Let $z\in\CC^\bullet$, and let $\alpha_z = \phi_z(\alpha)$, and $\beta_z = \phi_z(\beta)$; since $z$ does not lie on the closed negative real axis, there is a positive hyperbolic distance $d$ between the axes of $\alpha_z$ and $\beta_z$.

We renormalize $\alpha_z$ and $\beta_z$, by sending 0 to $+1$ and by requiring that the common perpendicular between the axes of $\alpha_z$ and $\beta_z$ have its endpoints at 0 and $\infty$, where $0$ is closer to the axis of $\alpha_z$. After this renormalization, $\alpha_z$ has fixed points at $\pm 1$, with $+1$ the positive fixed point if $m > 2$; $\beta_z$ has its positive fixed point at some point $t = |t|e^{i\theta}$, and its negative fixed point at $-t$. One sees at once that
$|t| = e^d$, and that $\theta$ is the angle
between the axes of $\alpha_z$ and $\beta_z$ obtained by parallel transport along the common perpendicular. (More precisely, $\theta$ is the angle from the ray pointing to $+1$ along the axis of $\alpha_z$ to the ray pointing to $t$ along the axis of $\beta_z$.) Then $t = e^{d + i\theta}$ is a geometric parameter for the subspace $\CC^\bullet\subset Hom(G , PSL(2,{\mathbb C}))$.

Using the invariance of the cross-ratio, we obtain

\begin{equation}\label{eq:1}
z= z(t) = (-t,+t;+1,-1) =  \frac{(t+1)^2}{(t-1)^2}.
\end{equation}

It was shown by Gehring, Marshall and Martin \cite{GMM:LowerBounds} that for $z\in {\mathcal T}$ there is a positive lower bound to the distance $d$. As stated above, this is also easy to see using the result of Chuckrow \cite{Chuckrow:Schottky}. It follows that $\sqrt{z}$,
which is positive for $z$ real and positive, is a holomorphic coordinate on $T$ that extends to the boundary, from which it follows that
\begin{equation}\label{Eq:tz}
t = t(z) = e^{d+i\theta} =  \frac{\sqrt{z}+1}{\sqrt{z}-1}
\end{equation}
is also a holomorphic coordinate on ${\mathcal T}$ that extends to the boundary.

We note that each point $t$ in the closure of ${\mathcal T}$ defines a homomorphism $\psi_t:G\to PSL(2,{\mathbb C})$, where $\psi_t = \phi_{z(t)}$. This homomorphism is in fact defined for all $t$ in the exterior of the unit disc.

\subsection{A boundary point   of the first kind}

There is a positive real number $L$ such that if $|t| > L$, then $t\in {\mathcal T}$.
For example, one can
 choose $L$ to be the  distance in the hyperbolic plane between the finite vertices of the triangle with angles $\pi/m$, $\pi/\ell$ and 0; equivalently, this is the distance between the fixed points
of the elliptic generators of orders $m$ and $\ell $ in the $(m,\ell ,\infty)$-triangle group.     Easy computation shows that
\begin{equation}
\cosh L = \frac{1 + \cos (\pi/m) \cos(\pi/\ell)}{\sin(\pi/m) \sin(\pi/\ell)}.
\end{equation}
For the  case
that especially interests us here, $m =  2$ and $\ell  = 3$, this triangle group is the elliptic modular group, and $L = \log\sqrt{3}$.

For each fixed angle, $\theta$, the ray $\arg t = \theta$ lies in ${\mathcal T}$ for $|t|$ sufficiently large, and does not lie even in the closure of ${\mathcal T}$ for $|t|$ sufficiently small. Hence there is some largest $t_\theta$ on this ray with $t_\theta$ on the boundary of ${\mathcal T}$.

For each fixed $g\in G $, the  entries in the matrix representing $\psi_t(g)$ are holomorphic functions of $t$. We note that $\psi_t(g)$ is parabolic only if the square of its trace is equal to 4. Hence there are at most countably many points in the closure of ${\mathcal T}$ (in fact, in the exterior of the unit disc), for which $\psi_t(g)$ can be parabolic. Since $G $ is countable, there are at most countably many points in the exterior of the unit disc for which some element of
$\psi_t(G )$ is parabolic.

We now choose the direction $\theta_0$ so that, for every point on the ray, $\arg t = \theta_0$, no element of the group $\psi_t(G)$ is parabolic. Let $t_0$ be the largest point on this ray, $\arg t = \theta_0$, where $t_0$ lies on the boundary of ${\mathcal T}$, and let $K = \psi_{t_0}(G)$. We will need below that there is a sequence of groups in ${\mathcal T}$ converging algebraically to $K$.

It follows from Chuckrow's theorem \cite{Chuckrow:Schottky} that $\psi_{t_0}$ is an isomorphism onto $K$; in particular, $K$ is an algebraic free product of cyclic groups of orders $m$ and $\ell $; we also know that no element of $K$ is parabolic.

\begin{prop} Every point of $\CP_1$  is a limit point of $K$.
In other words, $K$ is a Kleinian group of the first kind.\label{tiptop}
\end{prop}
\begin{proof}
Suppose that $K$ is of  the second kind; that is, suppose
that  it acts discontinuously at some point on the sphere at infinity.
Let $\Delta$ be a connected component of the set of discontinuity of $K$, and let $H$ be the stability subgroup of $\Delta$. Since $K$ is finitely generated, it follows from Ahlfors' finiteness theorem that $H$ is also finitely generated. Since $K$ is an algebraic free product of cyclic groups of orders $m$ and $\ell$, it follows that $H$ is an algebraic free product of a finite number of cyclic groups of orders $m$, $\ell$, and/or infinity.
The function group $H$ can be decomposed, using combination theorem amalgamated free products and HNN extensions, into {\it basic} groups; these are finitely
generated
subgroups of $H$, each containing no accidental parabolic element, and each having, as a Kleinian group in its own right, a simply connected invariant component of its set of discontinuity \cite {maskit}. It is also shown in \cite{maskit} that every basic group is either Fuchsian or quasifuchsian of the first kind, degenerate, Euclidean or finite; it is also well known that every such quasifucshian or degenerate group is isomorphic to a Fuchsian group of the first kind, where the isomorphism preserves
parabolic elements in both directions.

 Suppose $J$ is such a basic group. Since $K$ is a free product of cyclic groups, so is $J$; since $K$ contains no parabolic elements, neither does $J$. We conclude that $J$ cannot be Fuchsian, quasi-Fuchsian or degenerate, for every finitely-generated Fuchsian group of the first kind that is a free product of cyclic groups necessarily contains a parabolic element. We also conclude that $J$ is not Euclidean, for every Euclidean group contains parabolic elements. Hence $J$ is finite.

Since no basic subgroup of $H$ is either Fuchsian or quasi-Fuchsian, the set of discontinuity of $H$ is connected \cite{maskit}, implying both  that $H = K$ and that $\Delta$ is the full set of discontinuity of $K$. We now have that $K$ is a function group isomorphic to $G $.
By \cite{Maskit:iso}, it therefore follows that
 there is a quasi-conformal homeomorphism $w$ realizing the  isomorphism $\psi_{t_0}$.
 But  this contradicts our assumption that $t_0$ lies on the boundary of ${\mathcal T}$
and not in its interior. Thus $K$ cannot be of the second kind, and so must be
of the first kind, as claimed.
\end{proof}
\begin{thm}
For every $\epsilon > 0$, there is some  $t_\epsilon\in {\mathcal T}$  such that
the limit set of
$\psi_{t_\epsilon}(G )$
has  Hausdorff dimension  greater than $2-\epsilon$.
\end{thm}
\begin{proof}
Proposition \ref{tiptop} tells us that  the limit set of $K$ has Hausdorff dimension 2.
However, a result of Bishop and Jones \cite{bishopdim} asserts that if
$K_n\to K$ algebraically, then $\liminf \dim  \Lambda (K_n) \geq  \dim  \Lambda (K)$,
where $\dim  \Lambda$ denotes the Hausdorff dimension of the limit set of the
 relevant group.
Our result
therefore follows from the fact that, by construction,   $K$ is the algebraic limit of groups
of the form $\psi_t(G )$.
\end{proof}

\subsection{Scalar-Flat Orbifold Metrics} \label{warmup} 

We now specialize these conclusions to our original  case of 
$m=2$, and apply them to our  geometric setting. Doing so immediately
gives us an orbifold analogue of Theorem \ref{charlie}: 
 
\begin{prop} For each $\ell \geq 3$, 
there is a smooth  family $h_t$ of 
metrics on $X=(\ell -1)(S^1 \times S^3)$, $t\in [-1,1]$, such that 
\begin{itemize}
\item for each $t$, the metric $h_t$ is locally conformally flat and $D_\ell$-invariant; 
\item the metric $h_1$ is conformally related to a metric with   $s > 0$;  and 
\item the metric $h_{-1}$ is conformally related to a metric with  $s < 0$. 
\end{itemize}
 \end{prop}
 
 \begin{proof}
 Choose two points in $T$,  one corresponding to a group of 
 Hausdorff dimension $> 1$,  the other corresponding to a
 group with Hausdorff dimension $< 1$. Since the  manifold 
$T$ is connected, these two points can be joined by a smooth arc.
This arc  then corresponds to a  family of manifolds diffeomorphic
to $X= (\ell -1)(S^1 \times S^3)$, together with a smooth  family of flat conformal structures on them
and a smooth  family of actions of the dihedral group $D_\ell$
compatible with these conformal structures. 
The 
total space of this family is then diffeomorphic to 
$X \times I$ such a manner that the $D_\ell$ action is independent of the 
parameter. We now represent our conformal structures
by a smooth family of metrics. 
By adding all the pull-backs of these metrics with respect 
to the fixed $D_\ell$-action, we then obtain  representatives
$h_t$ which are $D_\ell$-invariant.  By
Lemma \ref{schauen}, 
the metric at one end-point, say $h_{-1}$,  is then conformal to a metric with 
negative scalar curvature, while the metric at the other end-point, say
$h_1$, is  conformal to a metric with 
positive scalar curvature.
\end{proof}

Now, for each $t$, let $\lambda_t$ be the lowest eigenvalue of the 
Yamabe Laplacian
$$\Delta_{h_t} + \frac{s_{h_t}}{6}$$
and recall that, by the minimum principle, any corresponding eigenfunction $u_t$ must be everywhere non-zero. 
 Hence  
$\lambda_t$ has multiplicity $1$, and so varies
continuously with $t$. However, we know that $\lambda_{-1} < 0$, and 
$\lambda_{+1} > 0$, so the intermediate value theorem predicts the existence of
some $t_0\in [-1,1]$ such that $\lambda_{t_0}=0$. Letting $u_{t_0}$ be the
corresponding positive eigenfunction, the conformally flat  metric 
$h= u_{t_0}^2h_{t_0}$ is then $D_\ell$-invariant, and 
has scalar curvature $s\equiv 0$. Thus:

\begin{prop}
For each $\ell$, the orbifold $V= [ (\ell -1)(S^1 \times S^3)]/D_\ell$ 
admits  locally conformally flat, scalar-flat orbifold metrics. Such metrics are
optimal, in the the sense of the natural  orbifold extension of Definition \ref{defopt}. 
\end{prop}

\subsection{Deformation Theory}

In order to carry out our gluing construction, we will want
to know that the deformation theory of anti-self-dual conformal structures
is formally unobstructed for each of the constructed  flat conformal structures
on $V$.  We will deduce this from an analogous statement about flat 
conformal structures on the connected sum $X=(\ell -1)(S^1 \times S^3)$.

To make this  precise, recall that the 
deformation theory of anti-self-dual conformal structures
on any $4$-orbifold $Y$ is governed by the elliptic complex \cite{DF} 
$$
0\to \Gamma (TY) \stackrel{{\mathcal L}}{\longrightarrow} \Gamma (\odot^2_0 \Lambda^1 )
 \stackrel{DW_+}{\longrightarrow} \Gamma (\odot^2_0 \Lambda^+ )\to 0$$
where $\odot^2_0$ denotes the trace-free symmetric product. 
Here
$\mathcal L$ computes  the Lie derivative of the  conformal metric along  vector fields, 
while $DW_+$ is the linearization  of the self-dual Weyl tensor at $g$. 
The deformation theory is unobstructed at a given conformal metric
iff $H^2$ of this complex vanishes. (We remark in passing that this amounts \cite{DF,eastsing2} to 
saying that the Kodaira-Spencer deformation theory of the corresponding twistor space
is unobstructed.)

\begin{prop} 
\label{vgood}
For    any flat conformal structure $[g]$ on the
orbifold $V=(S^4/\ZZ_2) \# (S^4/\ZZ_\ell)$,  the differential operator
$$DW_+: C^\infty (\odot^2_0 \Lambda^1 ) \to C^\infty (\odot^2_0 \Lambda^+ )$$
is surjective. In other words, the 
 deformation theory of ASD conformal structures on 
$V$ is unobstructed at $[g]$. 
\end{prop}
\begin{proof}
By ellipticity, it is enough to show that $\ker (DW_+)^*=0$. But   we saw in \S \ref{uniform}
that 
$V$ can be expressed as a global quotient $X/D_\ell$, where 
$$X= \underbrace{(S^3\times S^1) \# \cdots \# (S^3\times S^1)}_{\ell -1}~.$$
Thus any flat conformal structure on $V$ pulls back to a flat conformal structure on $X$, and any 
$\varphi \in \ker (DW_+)^*$ pulls back to an element 
of the cokernel of the corresponding operator $DW_+$  on $X$.
However, a Mayer-Vietoris argument due to  Eastwood and Singer shows  
\cite[Theorem 8.2]{loptimal} that the deformation of   ASD conformal structures 
 is unobstructed at every flat conformal structure on $X$. 
Since the pull-back of $\varphi$ to $X$ therefore vanishes, so
does $\varphi$ itself, and $DW_+$ is therefore surjective, as claimed. 
\end{proof}

\pagebreak

\section{ALE Metrics} 

In addition to the conformally flat orbifolds considered in the previous section,
our construction will crucially involve the use 
of  two  classes of examples of complete, non-compact anti-self-dual $4$-manifolds
which are asymptotically locally Euclidean (ALE).
The examples we will use are in fact all   scalar-flat and K\"ahler. 
After an explicit description of the metrics we will actually need, 
we then prove a general result (Theorem \ref{sesame}, \S \ref{flash}) concerning the deformation theory
of  ALE scalar-flat  K\"ahler manifolds.

\subsection{The Gibbons-Hawking Metrics}
\label{hyper} 

The first class of building blocks we will need consists of the Gibbons-Hawking gravitational 
instantons \cite{gibhawk}, which may be understood \cite{hitpoly} as Ricci-flat  K\"ahler
metrics on the minimal resolutions of the singular complex surfaces
$$xy=z^{\ell}.$$
Such a singular complex surface  can explicitly be identified with $\CC^2/\ZZ_\ell$, 
where 
$\ZZ_\ell \subset SU(2)$ is generated by
$$\left(\begin{array}{cc}e^{2\pi i\ell} & 0 \\0 & e^{-2\pi i\ell}\end{array}\right),$$
via the map 
\begin{eqnarray*}
\CC^2/\ZZ_\ell&\longrightarrow&\{~ (x,y,z)~~~|~~~ xy=z^\ell~\}\\
 \ZZ_\ell\cdot (u,v)&\longmapsto& (u^{\ell}, v^\ell ,uv ).
\end{eqnarray*}
The resolution of the singularity is accomplished by replacing the 
origin with a string of $\ell-1$ copies $\bcp_1$, each of self-intersection 
$-2$, and each only intersecting its succesor and/or predecessor, 
so that the pattern of intersections is dual to that indicated  by the Dynkin diagram $A_{\ell-1}$:
 \begin{center}
\mbox{
\beginpicture
\setplotarea x from 0 to 100, y from -10 to 10 
\put {$\bullet$} [B1] at 0 0 
\put {$\bullet$} [B1] at 15 0 
\put {$\bullet$} [B1] at 30 0 
\put {$\bullet$} [B1] at 45 0 
\put {$\bullet$} [B1] at 60 0 
\put{$\cdots$}[B1] at 73 0
\put {$\bullet$} [B1] at 85 0 
\put {$\bullet$} [B1] at 100 0 
{\setlinear  
\plot 0  3 60 3 /
\plot 85  3 100 3 /
}
\endpicture
}
\end{center}

Metrics in this family can be written  in closed form
by means of the {\em Gibbons-Hawking ansatz} \cite{gibhawk}. 
Choose   $\ell$ distinct point
$p_1, \ldots  , p_\ell$ in $\RR^3$, and set ${\mathcal U} = \{ p_1, \ldots , p_\ell\}$. 
Let $\rho_j : \RR^3 \to \RR$
be the Euclidean distance to $p_j$, and define  
$V: {\mathcal U} \to \RR^+$  by 
$$V = \sum_{j=1}^\ell \frac{1}{2\rho_j}.$$
Then $\Delta V =0$ on ${\mathcal U}$, so 
 the $2$-form
$$\star dV = V_x dy\wedge dz +V_y dz\wedge dx + V_z dy\wedge dz$$
is closed. Moreover, $[(\star dV)/2\pi] \in H^2 ({\mathcal U}, \ZZ)$,
since the integral of $\star dV$ on a small $2$-sphere centered at
 $p_j$ is $-2\pi$, and such  spheres generate $H_2 ({\mathcal U}, \ZZ )$.  
Let $\varpi: P\to {\mathcal U}$ be the circle bundle with first Chern class 
$[(\star dV)/2\pi]$. By the Chern-Weil theorem, there is a connection
$1$-form $\theta$ on $P$ with $d\theta = \varpi^* \star dV$. 
The Gibbons-Hawking metric on $P$ is then given by 
$$g_{GH} = V (dx^2 + dy^2 + dz^2) + V^{-1} \theta^{2}.$$
The metric-space completion $Y_\ell$ of $(P,g)$ is then 
a smooth Riemannian manifold, and is 
obtained by 
merely adding one point $\hat{p}_j$ for each each of the chosen 
points $p_j$, and we then have a smooth projection 
$$
\begin{array}{ccccc}Y_\ell & = & P & \cup & \{\hat{p}_1 , \ldots , \hat{p}_\ell\} \\
\downarrow &  & ~~\downarrow\varpi  &  & \downarrow \\
\RR^3  & = & {\mathcal U} & \cup & 
\{p_1 , \ldots , p_\ell\} \end{array}
$$
which, near each critical point $\hat{p}_j$,  is modelled on the 
map 
\begin{eqnarray*}
\RR^4 = \HH &\longrightarrow& \Im m~ \HH = \RR^3\\
q&\longmapsto& qi\bar{q} ~. 
\end{eqnarray*}
The Riemannian $4$-manifold $(Y_\ell ,g)$ is then 
ALE, with $\Gamma = \ZZ_\ell$, and is {\em hyper-Kahler},
in the sense that $\Lambda^+$ is trivialized by three linearly independent
parallel sections. It is also asymptotically locally Euclidean (ALE);
indeed, the region $\rho_1 > \max (|p_j - p_1|)$ in $Y_\ell$ may be identified
with the complement of a ball in $\RR^4/\ZZ_{\ell}$ in such a
manner such that $\varrho= \sqrt{2\ell \rho_1}$ becomes the Euclidean 
radius, and such that 
$$ g = g_{\mbox{\tiny Eucl}}+ O(\varrho^{-4}), ~~ \partial^k g = O(\varrho^{-4-k}).$$

\subsection{Line-Bundle Metrics}

The second family  of ALE metrics we will need for our construction consists of  the
so-called LeBrun metrics \cite{lpa} on the total spaces of negative-degree 
complex line bundles  $L\to \CP_1$
over the $2$-sphere. 
For each integer  $\ell \geq 1$, such a metric on the $c_1= -\ell$ line bundle may be 
obtained by taking the metric-space completion of the metric 
\begin{equation}
\label{lbg} 
g_{LB}=\frac{d\varrho^2}{1+\frac{\ell -2}{\varrho^2}- \frac{\ell -1}{\varrho^4}} 
+ \varrho^2\left(\sigma_1^2 + \sigma_2^2 +
\left[ 1+\frac{\ell -2}{\varrho^2}- \frac{\ell -1}{\varrho^4}\right]
\sigma_3^2
\right)
\end{equation}
on $(\CC^2- B)/\ZZ_\ell$, where $B\subset \CC^2$ is the closed unit ball, and
where the $\ZZ_\ell$ action is generated by scalar multiplication by $e^{2\pi i/\ell}$
on $\CC^2$. Here $\sigma_1$, $\sigma_2$, $\sigma_3$ is the standard $SU(2)$-invariant
orthonormal co-frame on the the unit sphere $S^3\subset \RR^4= \CC^2$,
and it is therefore easy to show that $g$ is $U(2)$-invariant.
When a $2$-sphere (corresponding to the zero section of the line bundle)
is added along $\varrho=1$, the metric extends smoothly and becomes complete. 
For our purposes, we will need only the cases $\ell \geq 2$. Note that 
when $\ell =2$, $g$ is  the celebrated Eguchi-Hanson metric, and is also given
by  the $\ell =2$ case of the Gibbons-Hawking construction described in 
\S \ref{hyper}. When  $\ell \neq 2$, however, 
the metric is no longer Ricci-flat; moreover, it is then just ALE of order $2$,  rather than of order
$4$:
$$ g_{LB} = g_{\mbox{\tiny Eucl}}+ O(\varrho^{-2}), ~~ \partial^k g = O(\varrho^{-2-k}).$$

\subsection{A Vanishing Theorem}
\label{flash}

The examples described  in the two previous sub-section are all 
ALE scalar-flat K\"ahler manifolds. However, 
that these examples by no means constitute an exhaustive list; there are many others 
 \cite{caldsing,kron}. 
 The purpose of the present section is to prove a
 general vanishing result pertaining to all such ALE spaces.
 We begin with a lemma which provides useful information on the 
  asymptotic structure of such manifolds. 

\begin{lem} \label{labby} 
Let $(Y,J,g)$ be a complete  scalar-flat K\"ahler manifold of real dimension $4$. 
Suppose, moreover, that $g$ is  ALE in the weak sense 
that 
$$g =  g_{\mbox{\tiny Eucl}}+ O(\varrho^{-\frac{3}{2}-\delta}) ,  
~~~\partial g =  O(\varrho^{-\frac{5}{2}-\delta}) $$
for some $\delta > 0$ in some real asymptotic coordinate system. Let
$U\subset Y$ be the domain of such an asymptotic coordinate system, 
and let $\tilde{U}\approx S^3 \times \RR$ be its universal cover. Then 
there is a non-singular  complex surface $(S,J)$ obtained by 
adding a $\CP_1$ of self-intersection $+1$ to $\tilde{U}$ at
infinity. Moreover,  small deformations of this holomorphic curve
pass through every point of $\tilde{U}$ in a neighborhood of infinity. 
 The complex surface  $(Y,J)$ therefore has only one end, and 
 is obtained from a non-singular, rational complex surface by removing 
a (typically singular) rational curve. 
Moreover,  $g$ is strongly ALE, of order $\geq 2$,
in the sense that in a (perhaps better) asymptotic coordinate chart, 
$$g =  g_{\mbox{\tiny Eucl}}+ O(\varrho^{-2}) ,  
~~~\partial^k g =  O(\varrho^{-2-k}) $$
for every positive integer $k$. 
\end{lem}
\begin{proof}
The fact  that the weak form of the ALE condition implies the stronger
form for a scalar-flat K\"ahler surface is proved in  \cite[Proposition 12]{chenlebweb}
by first showing that adding an extra twistor line $P$  at infinity to the twistor space of 
$\tilde{U}$ still yields a complex manifold $Z$. Let $S_0$  be the image of the section of 
the twistor projection $(Z-P) \to \tilde{U}$ given by $J$, and let
$\bar{S}_0$ similarly corresponding to $-J$. Then $S_0\cap \bar{S}_0 = \varnothing$,
and, because $g$ is scalar-flat K\"ahler,  Pontecorvo's theorem 
\cite{mano} tells us that $S_0\cup \bar{S}_0$ is the 
non-degenerate zero locus of a holomorphic section of $\kappa_Z^{-1/2}$. 
However, since $P\subset Z$ has complex codimension $2$, this section 
  uniquely extends  to all of $Z$ by Hartog's theorem. Since $S_0$ and $\bar{S}_0$ are
interchanged by the real structure of $Z$, this section is real, and its restriction 
to $P$, which is essentially  a section of ${\mathcal O}(2)$ on 
$\CP_1$, must either vanish identically, or vanish only at an antipodal   pair  of 
points. In the latter case, however, Pontecorvo's theorem  would say 
that $g$ extended to $\tilde{U}\cup \{ \infty\}$ as a scalar-flat K\"ahler metric,
contradicting the fact that $g$ is complete. The extension therefore vanishes
on all of $P$, and $S_0\cup \bar{S}_0 \cup P$ is therefore a (singular) complex hypersurface in 
$Z$. Let  $S$ be the irreducible component of this surface containing $S_0$.
Since every point of $P$ is either in the closure of $S_0$ or in the closure of 
$\bar{S_0}$, $P$ must meet $S$ in infinitely points. Hence  $P\subset S$,
and $S=S_0\cup P$.  Since a generic complex twistor line therefore
meets
$S$ in exactly one point, and since we can foliate any small region of 
$Z$ with a holomorphic family of 
 complex twistor lines which are not contained in 
$S$, 
the nullstellensatz \cite{GH} now implies that $S$ is actually non-singular.
Thus $S$ is a smooth complex surface obtained by adding $P\cong \CP_1$
to $S_0= (\tilde{U}, J)$. 
 Moreover, since
$P\subset Z$ has normal bundle ${\mathcal O}(1) \oplus {\mathcal O}(1)$,
whereas  $P\cdot S =1$, the adjunction formula implies that 
the normal bundle of $P\subset S$ has degree one. 

The normal bundle of the projective 
line $P\subset S$ is therefore   ${\mathcal O}(1)$. Since 
$H^1 (\CP_1 , {\mathcal O}(1))=0$,  Kodaira's theorem on 
deformations of compact complex submanifolds \cite{kodsub} guarantees that 
every element of $H^0 (\CP_1 , {\mathcal O}(1))$ arises from a
variation of $P\subset S$ through holomorphic curves,
and the union of these curves must therefore fill out some open 
 neighborhood $V\subset S$ of $P$. 
 
 The fundamental group $\Gamma$ of the end $U$  acts on 
 $\tilde{U}$ by isometries preserving $J$, and these
 extend to $\tilde{U} \cup \{ \infty\}$ as conformal isometries. 
 These in turn  lift to the twistor space $Z$ as biholomorphisms
 sending $S$ to $S$ and $P$ to $P$.  In particular, $\Gamma$ acts holomorphically 
 on $S$ in a manner sending $P$ to itself. Since the automorphism group of 
 the  first infinitesimal neighborhood of $P\subset S$ is $GL(2, \CC)$, the action of 
 the finite group  $\Gamma$ acts on $TS|_P$   is realized by that 
 of a finite subgroup of the maximal compact 
 subgroup $U(2)\subset GL(2, \CC)$. 
 If $\Gamma_0\subset \Gamma$ is the normal subgroup corresponding 
 to the center $U(1)\subset U(2)$, the quotient $\check{S} = S/\Gamma_0$ 
 is then a smooth complex surface, and $S\to \check{S}$ is then a branched cyclic cover 
  ramified along $P$. The induced action of $\check{\Gamma}=\Gamma/\Gamma_0$
  on $\check{S}$ 
 then has only cyclic isotropy groups with isolated fixed points,  so
 we can resolve all the singularities of $S/\Gamma= \check{S}/\check{\Gamma}$ 
  by replacing them with Hirzebruch-Jung strings \cite[Theorem 5.4]{bpv}. 
 Notice that the generic deformation of $P$ in $S$  
 will descend to $\check{S}$ as a $\CP_1$ of self-intersection $| \Gamma_0| \geq 1$,
 and that any such curve which avoids the finite fixed-point set of $\check{\Gamma}$
 will then become an immersed $\CP_1$ with positive normal bundle in the resolution of 
  $S/\Gamma$.

Applying this picture,  we may thus   compactify $Y$ as   a non-singular 
 compact complex surface $\hat{Y}$ by adding a tree of $\CP_1$'s to each
 end. But the compactification $\hat{Y}$ then   in particular contains
  holomorphically
 immersed $\CP_1$'s  with positive normal bundle passing through each
 point of in a
 region at each  end of $Y$. Since the 
  pluricanonical line-bundle $K^m$ and the
  cotangent bundle $\Omega^1$ pull back to any such curve as negative bundles, 
  all holomorphic sections of $K^m$ or $\Omega^1$
 must  vanish along all  such  curves, and since these curves
 sweep out an open subset of  $\hat{Y}$,   uniqueness of  analytic continuation 
 implies that $H^0 (\hat{Y}, \Omega^1)=0$ and $H^0 (\hat{Y}, {\mathcal O} (K^m))=0$ 
 for all $m > 0$. Moreover, the homology class of any such 
curve has positive self-intersection, as may be seen by 
moving a given $\CP_1$ within the appropriate Kodaira family; 
hence $b_+(\hat{Y}) \neq 0$. The Kodaira-Enriques classification \cite{bpv,GH} therefore
tells  us that 
 $\hat{Y}$ is a blow-up of either   $\CP_2$ or a Hirzebruch surface. 
 In  particular,
 $b_+(\hat{Y}) =1$. It follows that  $Y$ can only one end; otherwise, we would get two disjoint curves
 of positive self-intersction from any pair of ends, forcing $b_+$ to be larger. 
Hence    $Y$ is obtained
 from a  rational complex surface 
 $\hat{Y}$ by removing a (possibly  singular) embedded rational curve,  as claimed. 
  \end{proof}

With the ground now properly prepared, we are  now ready
to prove our main vanishing result.  Actually, 
 this very result was previously stated in \cite[Theorem 8.4]{kovsing}, but  the proof given
there  unfortunately assumes that the complex structure is asymptotically biholomorphic
to $\CC^2/\Gamma$,  and this,   alas, is     simply not true in general. 
However,   we shall now see that much the same  proof can be made 
to work in full generality with a modicum of extra care. 

\begin{thm} 
\label{sesame} 
Let $(Y,g)$ be any ALE  scalar-flat Kahler manifold
of real dimension $4$, 
and let $(\hat{Y}, [g])$ be its $1$-point conformal compactification,
considered as a compact  ASD orbifold. 
Then the deformation theory of ASD conformal structures on 
$\hat{Y}$ is unobstructed at $[g]$. 
That is,  the linearized self-dual Weyl curvature
$$DW_+ : C^\infty ( \odot^2_0T^*Y ) \to C^\infty  ( \odot^2_0\Lambda^+ )$$
is surjective at $[g]$. 
\end{thm}
\begin{proof}
Consider the adjoint equation 
\begin{equation}
\label{coke}
(DW_+)^*\varphi =0
\end{equation}
for an element  $\varphi \in \Gamma (\odot^2_0\Lambda^+)$
of the cokernel on the orbifold $\hat{Y}$. 
If ${\varphi^a}_{bcd}$ is simply considered as a totally trace-free 
 tensor field
 with the same symmetries as $W_+$, then
 (\ref{coke}) takes the explicit form
 \begin{equation}
\label{pepsi} 
\left(\nabla^b\nabla^d + \frac{1}{2}{r}^{bd}\right) {\varphi^a}_{bcd}=0,
\end{equation}
so   $(DW_+)(DW_+)^*$  is elliptic, and any distributional solution $\varphi$ on   
$\hat{Y}$ is necessarily smooth in the orbifold sense. Moreover, 
$DW_+$ belongs to an elliptic complex, so that $DW_+$ is surjective
iff (\ref{coke}) has no non-trivial solutions. 
Now the geometric nature of the construction of $DW_+$ implies
that it is conformally invariant; and since the $L^2$ inner product on 
such tensor fields ${\varphi^a}_{bcd}$ is conformally invariant, equation 
(\ref{coke}) is conformally invariant without any added conformal weight.
Thus the pull-back of $\varphi$  to $Y$  satisfies  both (\ref{coke})
and the  asymptotic fall-off conditions 
$$
|\varphi | = O (\rad^{-4}) , ~~ |\nabla^k \varphi | = O (\rad^{-4-k})
$$
at infinity, 
where $\rad$ denotes the Euclidean radius function in the 
asymptotic chart, and where the norms and derivatives are all  with respect to $g$. 
Working henceforth only on $(Y,g)$,  
our objective is therefore  to show that (\ref{coke})  and these boundary conditions 
suffice to imply that $\varphi$ vanishes.

To this end, think of $\varphi$ as belonging to $\mbox{End }(\Lambda^+)$,
and consider the  self-dual
$2$-form $\phi$ obtained by 
applying this endomorphism to the 
 K\"ahler form $\omega$; that is, let 
$$\phi= \varphi (\omega) \in \Gamma (\Lambda^+)$$
be the contraction of $\varphi$ with $\omega$. 
We then  recall that 
$$\Lambda^+ = \RR \omega \oplus \Re e (K)$$
on any K\"ahler surface, where $K=\Lambda^{2,0}$ is the
canonical line bundle. Thus
$$
\phi = f\omega + \alpha
$$
for some function $f$ and the real part $\alpha$ of some $(2,0)$-form. 
Then (\ref{coke}) implies \cite[Corollary 2.3]{lebsing1} that 
$$
d^-\delta ( f\omega + \alpha) -  f \rho =0,
$$
where $\rho$ is the Ricci form of $g$. 
Using the K\"ahler identities to rewrite  this    as 
$$
(d\Lambda d^c + d^c\Lambda d ) \alpha - 2dd^c f - \omega \Delta f - 2 f \rho = 0
$$
and then 
applying $dd^c$, we thus obtain  the fourth order scalar equation 
$$
(\Delta^2  + 2 r \cdot \nabla \nabla ) f = 0. 
$$
However, this  so-called Lichnerowicz equation can be re-written \cite{ls}  as 
$$\nabla_\nu\nabla^{\bar{\mu}} \nabla_{\bar{\mu}}\nabla^\nu f = 0,$$
or, in other words, as  
$$(\bar{\partial}\partial^\#)^*(\bar{\partial}\partial^\#)f=0,$$
where $\partial^\#$ means the $(1,0)$ component of the gradient. 
But since 
$$
|\nabla^k f | = O (\rad^{-4-k})
$$
two integrations by parts give us 
$$
\int_Y |(\bar{\partial}\partial^\#)f|^2 d\mu = \int_Y f (\bar{\partial}\partial^\#)^*(\bar{\partial}\partial^\#)f
d\mu =0. 
$$
Hence $\bar{\partial}\partial^\# f=0$, and 
$\partial^\#f = (\nabla f - iJ \nabla f)/2$ is thus a holomorphic vector field.
But this means that the Hamiltonian vector field $J\nabla f$ preserves both the 
K\"ahler form $\omega$ and the complex structure $J$. Therefore $\xi = J \nabla f$
also preserves $g$, and so 
is a Killing vector field. In particular, $\xi$ may be viewed as a conformal Killing field on 
$(\hat{Y}, [g])$. 
But our fall-off conditions say that $\xi = O (\rad^{-5} )$, and
in inverted coordinates about the  point at infinity, $\xi$ therefore has a zero of order $6$.
Since a conformal Killing field is completely determined by its $2$-jet at any 
point, it follows that 
 $\xi = J \nabla f$ vanishes identically. This shows that $f$ is constant, and our
fall-off therefore implies that $f\equiv 0$. 

It follows that the self-dual $2$-form $\phi = \alpha$ satisfies
$$d^- \delta \phi = 0.$$
Hence $d \delta \phi$ is self-dual, and  $|d\delta \phi |^2 d\mu =d\delta \phi \wedge d\delta \phi$.
Integration therefore gives 
$$\int _Y |d\delta \phi |^2 d\mu = \int _Y d\left( \delta \phi \wedge d\delta \phi \right) =0$$
by Stokes' theorem and the fact that $|\delta \phi \wedge d\delta \phi | = O (\rad^{-11})$. 
Thus $d\delta \phi =0$, and the self-dual $2$-form $\phi$ is therefore killed
by the Hodge Laplacian. But since our manifold has $s=0$ and $W_+=0$,
the Weitzenb\"ock formula for the Hodge Laplacian just says
$$(d+\delta)^2 \phi = \nabla^*\nabla \phi = 0.$$
Since $\phi \cdot \nabla \phi = O (\rad^{-9})$, integration by parts therefore
says that 
$$
\int_Y |\nabla \phi |^2 d\mu = \int_Y \langle \phi , \nabla^* \nabla \phi \rangle  d\mu = 0
$$
so that $\nabla \phi=0$ is parallel, and so in particular, $|\phi |$ is constant. But
$|\phi |\to 0$ at infinity, so it follows that $\phi \equiv 0$.

We must therefore have $\varphi = \beta + \overline{\beta}$ for some 
$\beta \in \Gamma ( \kappa^{\otimes 2})$, where $\kappa= \Lambda^{2,0}$ is the canonical
line bundle. Because
the Ricci tensor belongs to $\Lambda^{1,1}$, equation (\ref{pepsi}) now
simplifies to become $$\nabla^\kappa\nabla^\mu \beta_ {\kappa  \lambda \mu\nu}=0.$$
Letting  $\psi \in \Gamma (\Omega^1 \otimes \kappa)$ be defined by 
$$\psi_{\kappa  \lambda \nu} = \nabla^\mu  \beta_{ \kappa  \lambda \mu \nu},$$
we then have $\partial^* \psi =0$, so $\overline{\partial}\psi =0$, and $\psi$
is holomorphic.

 Now let $U\subset Y$ be an asymptotic region of 
$Y$, and let $\tilde{U} \approx S^3 \times \RR$ be its universal cover. 
We have  shown, in  Lemma \ref{labby}, that 
we can  form a complex surface $(S,J)$ by adding a holomorphic curve $P\cong \CP_1$ of self-intersection 
$+1$  at infinity; moreover, deformations of   $P$ then 
gives us  rational holomorphic curves of self-intersection 
$+1$ sweeping out an open neighborhood $V\subset S$ of $P$.
The asymptotic fall-off of $\phi$ now implies that the pull-back  of $\psi$ 
to $\tilde{U}$ extends
continuously, and hence holomorphically, to $S$. 
 But $V$ is swept out by $\CP_1$'s with 
 normal bundle ${\mathcal O}(1)$. Since the restriction of 
 $\Omega^1 \otimes \kappa$ to such a curve is isomorphic to 
 ${\mathcal O} (-5) \oplus {\mathcal O} (-4)$, it follows that 
 $\psi$ vanishes identically on $V$, hence on its image in $Y$,  and thus on all of $Y$
 by the uniqueness of analytic continuation. 
 
 The vanishing of $\psi$ now tells us that 
 $$g^{\kappa \bar{\pi}} \nabla_{\bar{\pi}} \beta_{\kappa \mu \nu \lambda}
 =  \nabla^{\kappa} \beta_{\kappa \mu \nu \lambda}=0$$
 so we now have $\overline{\partial} \beta=0$, and $\beta$ is 
 itself holomorphic. But once again, our fall-off conditions imply that the pull-back of 
 $\beta$ to $\tilde{U}$ extends holomorphically
 to $S$. Since the restriction of $\kappa^{\otimes 2}$ to a $\CP_1$ of self-intersection 
 $+1$  is isomorphic to 
 ${\mathcal O}(-6)$, we thus see that $\beta$ must vanish identically
 on  the open set $V\subset S$ swept out by 
 such curves, and hence on all of $Y$ by analytic continuation. 
 Thus $\varphi = \beta+ \bar{\beta} \equiv 0$, and the cokernel of 
 $DW_+$ is trivial, as claimed. 
\end{proof}

\section{The Gluing Construction}

We are now in a position to construct the desired family of anti-self-dual metrics on 
$k\overline{\CP}_2$, $k\geq 5$. The basic tool we will need is the following gluing result 
\cite[Theorem C]{kovsing}; cf. \cite{floer, tasd}. 

\begin{prop}[Floer {\em et al.}] \label{gluing} 
Let $(V_j,[g_j])$ be a finite collection of smooth compact oriented $4$-dimensional
orbifolds with anti-self-dual conformal structure, and suppose that
these conformal structures are {\em deformation-unobstructed}  in the sense that 
$$DW_+ : C^\infty ( \odot^2_0T^*Y ) \to C^\infty  ( \odot^2_0\Lambda^+ )$$
is surjective for each of these spaces.  Further, suppose we are given 
 {\em gluing data}, consisting of a finite collection of pairs $(p_\alpha, q_\alpha) 
 \in \coprod  V_j$ of distinct points, together with a collection of  linear maps 
 $$\Lambda_\alpha: T_{p_\alpha} \to T_{q_\alpha}$$
 between the orbifold tangent space of $\coprod  V_j$
 at these points, such that $\Lambda_\alpha$
 is an orientation-reversing conformal isometry, and such that, 
 for each $\alpha$,  either 
\begin{itemize}
\item $p_\alpha$ and $q_\alpha$ are both 
manifold (non-singular) points of $ \coprod  V_j$; or else 
\item $p_\alpha$ and $q_\alpha$ are both 
orbifold (singular) points, the relevant orbifolds are locally modelled on 
 $T_{p_\alpha} /\Gamma_\alpha$ and  $T_{q_\alpha} /\Gamma_\alpha$ 
 for the same finite group $\Gamma_\alpha$,  and the 
 relevant representations of  $\Gamma_\alpha$ are intertwined
 by  $\Lambda_\alpha$. 
\end{itemize}
Let $M$ be the new orbifold obtained from $\coprod  V_j$
by first removing a small ball around each $p_\alpha$ and $q_\alpha$,
and then identifying the resulting boundaries via the orientation-reversing 
diffeomorphism indicated by $\Lambda_\alpha$. 
Then $V$ admits a  family of anti-self-dual metrics $[g_{\mathfrak u}]$,
smoothly  (but perhaps redundantly)  parameterized
by  deformation-unobstructed ASD conformal structures
on  $\coprod V_j$ and choices of gluing data. 
Moreover, these conformal structures resemble those specified on each $V_j$
in the following sense: there are sequences of these $g_{u_I}$ of these
metrics and diffeomorphisms $\Phi_I : (V_j - \{ p_\alpha , q_\alpha \}) \hookrightarrow
M$ such that the conformal metrics $[\Phi_I^* g_{u_I}]$ converge to  $[g_j]$ in  the 
$C^2$ topology 
on every compact subset ${\mathbf K} \subset (V_j - \{ p_\alpha , q_\alpha \})$. 
\end{prop}

We remark that the redundancy of parameterization 
alluded to above  only occurs if some of the $V_j$ carry 
conformal Killing vector fields, and   can  be eliminated by adding marked  points
to the picture.  
Note that this parameterization by no means describes the entire
moduli space of anti-self-dual metrics on $V$, about which
still know all too little; rather, it simply says that this moduli space
has a preferred end,  where we know quite a bit.

As sketched in \S \ref{stratego}, we now consider the orbifolds
$V_1= S^4/\ZZ_2$,  and $V_2= S^4/\ZZ_\ell$, $\ell \geq 3$, 
together with two copies 
$V_3$ and $V_4$ of the one-point compactification of the 
Eguchi-Hanson space $Y_2$, one copy $V_5$ of the compactified 
$A_{\ell-1}$ Gibbons-Hawking space $Y_\ell$, and one copy
$V_6$ of the compactification of the $c_1= -\ell$ line bundle over $\CP_1$
equipped 
with the LeBrun metric. We consider the ordinary connect sum
$V=V_1\# V_2$  centered at manifold points, and
 then take the generalized connect sum 
 $$M_\ell= V\#_{\ZZ_2} V_3\#_{\ZZ_2} V_4 \#_{\ZZ_\ell} V_5 \#_{\ZZ_\ell} V_6$$
 centered at appropriate pairs of singular points. Because we have eliminated 
 all the singular points, this gives 
 a smooth compact $4$-manifold, and it is moreover easy to see by 
 Seifert-van Kampen and 
 Mayer-Vietoris that $M_\ell$ is simply connected, and has $b_2=\ell + 2$. 
 By Proposition \ref{vgood} and Theorem \ref{sesame}, the hypotheses
 of Proposition \ref{gluing} are fulfilled, and we thus have

\begin{cor} \label{core} 
The gluing strategy proposed in \S \ref{stratego} 
does actually give rise, via Proposition \ref{gluing}, 
 to a connected  family of anti-self-dual metrics on the smooth
compact simply connected manifold $M_\ell$ with $b_2 = \ell + 2$, 
$\ell \geq 3$. Moreover, for any compact subset ${\mathbf K}$ of the complement
of the four singular points of $V$, there are embeddings ${\mathbf K}\hookrightarrow 
M$ on which some of these metrics are arbitrarily close, in the 
$C^2$ sense, to any one of the conformally flat metrics $h$ of
\S \ref{warmup}. On the other hand, one can also construct some 
metrics of this family by instead applying Proposition \ref{gluing}
to the orbifolds $V_1, \ldots  , V_6$, without forming the connect sum 
$V= V_1\# V_2$ as an intermediate step. 
\end{cor}

Note that, in the complex 
coordinate on ${\mathcal T}$ described in \S \ref{coordinate} , 
direct construction from $V_1, \ldots  , V_6$ corresponds to the region
$z \approx 1$, where $V= V_1\# V_2$  has an extremely long neck separating the 
two orbifold points  of type  $\ZZ_\ell$ from the two singularities of type $\ZZ_2$. 
\subsection{Negative Scalar Curvature}

Let us now show that some of the
conformal structures of Corollary \ref{core} have representatives
of negative scalar curvature.

To see this, recall that if   the 
lowest eigenvalue of the Yamabe Laplacian of a compact Riemannian 
manifold is negative, the metric can then be conformally rescaled  to yield a metric of 
negative scalar curvature simply by multiplying it by the square of the corresponding
eigenfunction. In dimension $4$, it therefore suffices to show that 
$$ \inf_{u\in L^2_1} \frac{\int [6|\nabla u|^2 + s u^2 ] d\mu}{\int u^2d\mu}$$
is negative; in particular, it suffices to display some smooth test function 
$u$ such that 
$$\int_M [6|\nabla u|^2 + s u^2 ] d\mu_g < 0.$$
Note that this test function need not  be positive everywhere. 

We saw in \S \ref{warmup} that 
$V= V_1\# V_2$ admits  conformally flat orbifold metrics of negative scalar curvature.
By rescaling , we can therefore find such metrics with scalar curvature $s< -1$. 
Let  $u$  be a function $V$ which is 
identically zero on the $(\varepsilon/2)$-balls about the four oribifold
singularities, equal to $1$ outside the corresponding $(\varepsilon)$-balls,
and has $|\nabla u|< 3/\varepsilon$ in the transition annuli.  We further take 
$\varepsilon\in (0,1)$ to be small enough so that the total volume of these
four  (orbifold)  
$\varepsilon$-balls is  less than $C\varepsilon^4$ where e.g. 
 $C=2\pi^2$. With respect to the background metric $h$, we then 
 have 
 $$\int_Y [6|\nabla u|^2 + s u^2 ] d\mu_h < 54 C \varepsilon^2 - [\Vol (Y) - C \varepsilon^4] < 
 55 C \varepsilon^2 - \Vol (Y)$$
 which is negative  for   $\varepsilon$ is sufficiently small. 
 For such an $\varepsilon$, 
let ${\mathbf K}\subset V$ be the complement of the four $(\varepsilon/2)$-balls
about the orbifold points. 
  Corollary \ref{core} then tells us that there exists a family of 
 representative metrics $g_j$ for some of our conformal structures 
 on $M_\ell$, together with a sequence of diffeomorphisms 
 $\Phi_j : {\mathbf K} \hookrightarrow M_\ell$, such that $\Phi_j^* g_j \to h$ in $C^2$ 
 on ${\mathbf K}$. In particular, 
  $$\int_{\mathbf K} [6|\nabla u|^2 + s u^2 ] d\mu_{\Phi_j^* g_j} < 0$$
  for sufficiently large $j$. For some such $j$, we  can thus construct
  a test function $\hat{u}$ on $M_\ell$ by extending the push-forward
 $u \circ \Phi_j^{-1}$ by zero, and this test function 
 then satisfies
  $$\int_{M_\ell} [6|\nabla \hat{u}|^2 + s \hat{u}^2 ] d\mu_{g_j} < 0.$$
  Thus 
  $(M,[g_j])$ can  be
 conformally rescaled to yield a metric of negative scalar curvature, and we have:

\begin{lem} 
\label{trombone} 
For each $\ell\geq 3$, 
some of the anti-self-dual conformal classes  on $M_\ell$ constructed in Corollary 
\ref{core} can be represented by Riemannian 
metrics of negative scalar curvature. 
\end{lem}

Note that the gist of this argument is  quite soft, and widely applicable. The moral is that negative
scalar curvature on one summand can be used to generate negative scalar curvature
on a generalized connected sum; cf. \cite{joyscal}. 

\subsection{Positive Scalar Curvature}

We next show that other conformal classes in our connected family
have representatives of positive scalar curvature. Results of 
 Joyce \cite{joyscal}   indicate that this can be done by 
 a direct argument in the spirit of the proof of Lemma \ref{trombone};
however,    the positive case in this setting 
 is technically quite delicate. Instead, we will proceed here via
 an entirely different route,  exploiting the considerable body of
 information concerning the spaces in question. 
 
 To this end, let us instead consider the 
 manifolds 
 $$N= V_3 \#_{\ZZ_2} V_1 \#_{\ZZ_2} V_4 =  V_3 \#_{\ZZ_2} V_4$$
 and
 $$N_\ell = V_5  \#_{\ZZ_\ell} V_2 \#_{\ZZ_\ell} V_6  = V_5 \#_{\ZZ_\ell} V_6.$$
 In both these cases, Proposition \ref{gluing} allows us to 
 deduce the existence of ASD conformal structures built out of the 
 given building blocks. However, one knows families of explicit ASD metrics of 
 with the required degenerations on  
 these manifolds;  some of these metrics 
 admit semi-free circle actions, and arise from the hyperbolic ansatz  \cite[pp. 235--237]{mcp2},
 and the relevant degenerations involve bringing the centers together until
 they collide.
In particular, this explicitly identifies $N$ as the smooth manifold
$2\overline{\CP}_2$, and $N_\ell$ as $\ell\overline{\CP}_2$. 
Moreover, all of these explicit conformal classes are known 
\cite[p. 235]{mcp2}  to admit representatives of positive scalar curvature. 
Some of the metrics of our family are therefore obtained by 
directly applying Proposition \ref{gluing} to glue together some 
of these 
positive-scalar-curvature ASD metrics  on $N$ and $N_\ell$. 
For such ordinary connect sums, however,
the gluing procedure can be realized by the Donaldson-Freedman
construction \cite{DF}, and in this context, a theorem of Atiyah on the
 the Penrose Transforms of Yamabe Green's  functions allows one  to show
 \cite{loptimal,kalafat}
that connected sums of positive-scalar-curvature ASD manifolds with small
gluing parameters again have representatives of positive scalar curvature. 
This shows  the following:

\begin{lem} 
\label{trombino} 
For each $\ell\geq 3$, $M_\ell$ is diffeomorphic to $(\ell+2)\overline{\CP}_2$, and 
some of the anti-self-dual conformal classes  on $M_\ell$ constructed in Corollary 
\ref{core} can be represented by Riemannian 
metrics of positive scalar curvature. 
\end{lem}

\subsection{The Main Theorems} 

Combining Lemmas \ref{trombone} and \ref{trombino}, we now have a smooth
connected 
family of anti-self-dual metrics on the connected sum  $(\ell+2)\overline{\CP}_2$,
$\ell \geq 3$, such that some of the conformal classes are represented 
by metrics of positive scalar curvature, and others are represented
by metrics of negative scalar curvature. 
Choosing a smooth path in our parameter space which joins 
a conformal class of the first type to one of the second, 
 and setting $k = \ell +2$, we 
have therefore proved

\setcounter{main}{1}
\begin{main}
For any integer $k\geq 5$, the smooth  oriented $4$-manifold 
$k\overline{\bcp}_2$ 
admits
a smooth $1$-parameter family of anti-self-dual  Riemannian  metrics $g_t$, 
$t\in [-1,1]$,
such that $[g_1]$ contains a metric of positive scalar curvature, while
$[g_{-1}]$ contains a metric of negative scalar curvature.
\end{main}

Now let $\lambda_t$ be the smallest eigenvalue of the  Yamabe Laplacian
$$ \Delta+ \frac{s}{6}$$ for the metric $g_t$.
By the minimum principle, any corresponding eigenfunction $u_t$ must be everywhere non-zero. 
 Hence  
$\lambda_t$ has multiplicity $1$, and so varies
continuously with $t$. However, we know that $\lambda_{-1} < 0$, and 
$\lambda_{+1} > 0$, so the intermediate value theorem predicts the existence of
some $t_0\in [-1,1]$ such that $\lambda_{t_0}=0$. Letting $u_{t_0}$ be the
corresponding positive eigenfunction, the ASD metric 
$g= u_{t_0}^2g_{t_0}$ then has scalar curvature $s\equiv 0$. This
proves

\begin{thm}
For each $k\geq 5$, the smooth compact oriented $4$-manifold
$$k\overline{\bcp}_2=\underbrace{\overline{\bcp}_2\# \cdots \# \overline{\bcp}_2}_k$$ 
admits scalar-flat anti-self-dual Riemannian metrics. In particular, each of these 
spaces carries optimal metrics, in the sense of Definition \ref{defopt}. 
\end{thm} 

Combining this with the results of Yau \cite{yauma} on $K3$ and 
Rollin-Singer \cite{rollsing} on $\CP_2 \# k\overline{\bcp}_2$, we 
 have therefore proved Theorem \ref{able}, as promised. 

\vfill

\noindent
{\footnotesize Authors' address: Mathematics Department, SUNY, Stony Brook, NY 11794-3651.}

\noindent
{\footnotesize Keywords: $4$-manifold, anti-self-dual metric, scalar curvature, Kleinian group.}

\noindent
{\footnotesize Mathematics Subject Classification: 53C25 (primary), 30F40 (secondary).}

\pagebreak


\begin{thebibliography}{10}

\bibitem{A-B:Beltrami}
{\sc L.~Ahlfors and L.~Bers}, {\em Riemann's mapping theorem for variable
  metrics}, Ann. of Math. (2), 72 (1960), pp.~385--404.

\bibitem{AHS}
{\sc M.~F. Atiyah, N.~J. Hitchin, and I.~M. Singer}, {\em Self-duality in
  four-dimensional {R}iemannian geometry}, Proc. Roy. Soc. London Ser. A, 362
  (1978), pp.~425--461.

\bibitem{bpv}
{\sc W.~Barth, C.~Peters, and A.~V. de~Ven}, {\em Compact Complex Surfaces},
  Springer-Verlag, 1984.

\bibitem{bercent}
{\sc M.~Berger}, {\em Riemannian Geometry During the Second Half of the
  Twentieth Century}, American Mathematical Society, Providence, RI, 2000.
\newblock Reprint of the 1998 original.

\bibitem{Bers:BdrysI}
{\sc L.~Bers}, {\em On boundaries of {T}eichm\"uller spaces and on {K}leinian
  groups. {I}}, Ann. of Math. (2), 91 (1970), pp.~570--600.

\bibitem{bishopdim}
{\sc C.~J. Bishop and P.~W. Jones}, {\em Hausdorff dimension and {K}leinian
  groups}, Acta Math., 179 (1997), pp.~1--39.

\bibitem{caldsing}
{\sc D.~M.~J. Calderbank and M.~A. Singer}, {\em Einstein metrics and complex
  singularities}, Invent. Math., 156 (2004), pp.~405--443.

\bibitem{chenlebweb}
{\sc X.~X. Chen, C.~LeBrun, and B.~Weber}, {\em On conformally K\"ahler,
  Einstein manifolds}.
\newblock e-print arXiv:0705.0710, 2007.

\bibitem{Chuckrow:Schottky}
{\sc V.~Chuckrow}, {\em On {S}chottky groups with applications to Kleinian
  groups}, Ann. of Math. (2), 88 (1968), pp.~47--61.

\bibitem{DF}
{\sc S.~Donaldson and R.~Friedman}, {\em Connected sums of self-dual manifolds
  and deformations of singular spaces}, Nonlinearity, 2 (1989), pp.~197--239.

\bibitem{donaldson}
{\sc S.~K. Donaldson}, {\em An application of gauge theory to four-dimensional
  topology}, J. Differential Geom., 18 (1983), pp.~279--315.

\bibitem{eastsing2}
{\sc M.~G. Eastwood and M.~A. Singer}, {\em The {F}r\"ohlicher [{F}r\"olicher]
  spectral sequence on a twistor space}, J. Differential Geom., 38 (1993),
  pp.~653--669.

\bibitem{EH}
{\sc T.~Eguchi and A.~J. Hanson}, {\em Self-dual solutions to {E}uclidean
  gravity}, Ann. Physics, 120 (1979), pp.~82--106.

\bibitem{floer}
{\sc A.~Floer}, {\em Self-dual conformal structures on {$\ell{\bf C}{\rm P}\sp
  2$}}, J. Differential Geom., 33 (1991), pp.~551--573.

\bibitem{freedman}
{\sc M.~Freedman}, {\em On the topology of 4-manifolds}, J. Differential Geom.,
  17 (1982), pp.~357--454.

\bibitem{GMM:LowerBounds}
{\sc F.~W. Gehring, T.~H. Marshall, and G.~J. Martin}, {\em The spectrum of
  elliptic axial distances in {K}leinian groups}, Indiana Univ. Math. J., 47
  (1998), pp.~1--10.

\bibitem{gibhawk}
{\sc G.~W. Gibbons and S.~W. Hawking}, {\em Classification of gravitational
  instanton symmetries}, Comm. Math. Phys., 66 (1979), pp.~291--310.

\bibitem{GH}
{\sc P.~Griffiths and J.~Harris}, {\em Principles of Algebraic Geometry},
  Wiley-Interscience, New York, 1978.

\bibitem{hitpoly}
{\sc N.~J. Hitchin}, {\em Polygons and gravitons}, Math. Proc. Cambridge
  Philos. Soc., 85 (1979), pp.~465--476.

\bibitem{joyscal}
{\sc D.~Joyce}, {\em Constant scalar curvature metrics on connected sums}, Int.
  J. Math. Math. Sci.,  (2003), pp.~405--450.

\bibitem{kalafat}
{\sc M.~Kalafat}, {\em Scalar curvature and connected sums of self-dual
  4-manifolds}.
\newblock e-print arXiv:math/0611769, 2006.

\bibitem{klp}
{\sc J.~Kim, C.~LeBrun, and M.~Pontecorvo}, {\em Scalar-flat {K}\"ahler
  surfaces of all genera}, J. Reine Angew. Math., 486 (1997), pp.~69--95.

\bibitem{kodsub}
{\sc K.~Kodaira}, {\em A theorem of completeness of characteristic systems for
  analytic families of compact submanifolds of complex manifolds}, Ann. of
  Math. (2), 75 (1962), pp.~146--162.

\bibitem{kovsing}
{\sc A.~Kovalev and M.~Singer}, {\em Gluing theorems for complete
  anti-self-dual spaces}, Geom. Funct. Anal., 11 (2001), pp.~1229--1281.

\bibitem{kron}
{\sc P.~B. Kronheimer}, {\em Instantons gravitationnels et singularit\'es de
  {K}lein}, C. R. Acad. Sci. Paris S\'er. I Math., 303 (1986), pp.~53--55.

\bibitem{kuiper}
{\sc N.~H. Kuiper}, {\em On conformally-flat spaces in the large}, Ann. of
  Math. (2), 50 (1949), pp.~916--924.

\bibitem{laf}
{\sc J.~Lafontaine}, {\em Remarques sur les vari\'et\'es conform\'ement
  plates}, Math. Ann., 259 (1982), pp.~313--319.

\bibitem{lsd}
{\sc C.~LeBrun}, {\em On the topology of self-dual $4$-manifolds}, Proc. Amer.
  Math. Soc., 98 (1986), pp.~637--640.

\bibitem{lpa}
\leavevmode\vrule height 2pt depth -1.6pt width 23pt, {\em Counter-examples to
  the generalized positive action conjecture}, Comm. Math. Phys., 118 (1988),
  pp.~591--596.

\bibitem{mcp2}
\leavevmode\vrule height 2pt depth -1.6pt width 23pt, {\em Explicit self-dual
  metrics on {${\mathbb {C}}{\mathbb {P}}\sb 2\#\cdots\#{\mathbb {C}}{\mathbb
  {P}}\sb 2$}}, J. Differential Geom., 34 (1991), pp.~223--253.

\bibitem{loptimal}
\leavevmode\vrule height 2pt depth -1.6pt width 23pt, {\em Curvature
  functionals, optimal metrics, and the differential topology of
  four-manifolds}, in Different {F}aces of {G}eometry, Int. Math. Ser. (N. Y.),
  Kluwer/Plenum, New York, 2004, pp.~199--256.

\bibitem{ls}
{\sc C.~LeBrun and S.~R. Simanca}, {\em Extremal {K}\"ahler metrics and complex
  deformation theory}, Geom. Funct. Anal., 4 (1994), pp.~298--336.

\bibitem{lebsing1}
{\sc C.~LeBrun and M.~Singer}, {\em Existence and deformation theory for
  scalar-flat {K}\"ahler metrics on compact complex surfaces}, Invent. Math.,
  112 (1993), pp.~273--313.

\bibitem{lebsing2}
\leavevmode\vrule height 2pt depth -1.6pt width 23pt, {\em A {K}ummer-type
  construction of self-dual {$4$}-manifolds}, Math. Ann., 300 (1994),
  pp.~165--180.

\bibitem{Maskit:BdrysII}
{\sc B.~Maskit}, {\em On boundaries of {T}eichm\"uller spaces and on {K}leinian
  groups. {II}}, Ann. of Math. (2), 91 (1970), pp.~607--639.

\bibitem{Maskit:iso}
\leavevmode\vrule height 2pt depth -1.6pt width 23pt, {\em Isomorphisms of
  function groups}, J. Analyse Math., 32 (1977), pp.~63--82.

\bibitem{maskit}
\leavevmode\vrule height 2pt depth -1.6pt width 23pt, {\em Kleinian Groups},
  Springer-Verlag, Berlin, 1988.

\bibitem{shinpat}
{\sc S.~Nayatani}, {\em Patterson-{S}ullivan measure and conformally flat
  metrics}, Math. Z., 225 (1997), pp.~115--131.

\bibitem{pnlg}
{\sc R.~Penrose}, {\em Nonlinear gravitons and curved twistor theory}, General
  Relativity and Gravitation, 7 (1976), pp.~31--52.

\bibitem{mano}
{\sc M.~Pontecorvo}, {\em On twistor spaces of anti-self-dual {H}ermitian
  surfaces}, Trans. Amer. Math. Soc., 331 (1992), pp.~653--661.

\bibitem{rollsing}
{\sc Y.~Rollin and M.~Singer}, {\em Non-minimal scalar-flat {K}\"ahler surfaces
  and parabolic stability}, Invent. Math., 162 (2005), pp.~235--270.

\bibitem{schyaudim}
{\sc R.~Schoen and S.-T. Yau}, {\em Conformally flat manifolds, {K}leinian
  groups and scalar curvature}, Invent. Math., 92 (1988), pp.~47--71.

\bibitem{tasd}
{\sc C.~H. Taubes}, {\em The existence of anti-self-dual conformal structures},
  J. Differential Geom., 36 (1992), pp.~163--253.

\bibitem{trud}
{\sc N.~Trudinger}, {\em Remarks concerning the conformal deformation of
  metrics to constant scalar curvature}, Ann. Scuola Norm. Sup. Pisa, 22
  (1968), pp.~265--274.

\bibitem{yauruled}
{\sc S.~T. Yau}, {\em On the curvature of compact {H}ermitian manifolds},
  Invent. Math., 25 (1974), pp.~213--239.

\bibitem{yauma}
\leavevmode\vrule height 2pt depth -1.6pt width 23pt, {\em On the {R}icci
  curvature of a compact {K}\"ahler manifold and the complex {M}onge-{A}mp\`ere
  equation. {I}}, Comm. Pure Appl. Math., 31 (1978), pp.~339--411.

\bibitem{zhouthesis}
{\sc J.~Zhou}, {\em Connected Sums of Self-Dual Orbifolds}, PhD thesis, State
  University of New York at Stony Brook, 1995.

\end{thebibliography}
\end{document}